\documentclass[10pt]{article}%
\usepackage{amsmath}
\usepackage{amsfonts}
\usepackage{mathrsfs}
\usepackage{amssymb,color}
\usepackage[linkcolor=black,anchorcolor=black,citecolor=black]{hyperref}
\usepackage{graphicx}
\numberwithin{equation}{section}
\usepackage[body={15.5cm,21cm}, top=3cm]{geometry}%
\usepackage[numbers,sort&compress]{natbib}
\providecommand{\U}[1]{\protect\rule{.1in}{.1in}}
\providecommand{\U}[1]{\protect \rule{.1in}{.1in}}
\newtheorem{theorem}{Theorem}[section]

\newtheorem{definition}[theorem]{Definition}

\newtheorem{lemma}[theorem]{Lemma}

\newtheorem{proposition}[theorem]{Proposition}
\newtheorem{remark}[theorem]{Remark}
\newtheorem{assumption}[theorem]{Assumption}

\newenvironment{proof}[1][Proof]{\noindent \textbf{#1.} }{\  \rule{0.5em}{0.5em}}

\def \P{\mathsf{P}}
\def \E{\mathsf{E}}
\def \hE {\hat{\mathbb E}}
\begin{document}
	\title{Mean Reflected Backward Stochastic Differential Equations Driven by $G$-Brownian Motion with Double Constraints}
	\author{Wei He \thanks{Research Center for Mathematics and Interdisciplinary Sciences, Shandong University, Qingdao 266237, Shandong, China. hew@sdu.edu.cn.}\and Hanwu Li\thanks{Research Center for Mathematics and Interdisciplinary Sciences, Shandong University, Qingdao 266237, Shandong, China. lihanwu@sdu.edu.cn.}
	\thanks{Frontiers Science Center for Nonlinear Expectations (Ministry of Education), Shandong University, Qingdao 266237, Shandong, China.} }
	\date{}
	\maketitle
	\begin{abstract}
	In this paper, we study the backward stochastic differential equations driven by $G$-Brownian motion with double mean reflections, which means that the constraints are made on the law of the solution. Making full use of the  backward Skorokhod problem with two nonlinear reflecting boundaries and the fixed-point theory, the existence and uniqueness  of solutions are established. We also consider the case  where the coefficients satisfy a non-Lipschitz condition using the Picard iteration argument only for the $Y$ component. Moreover, some basic properties including a new version of comparison theorem and connection with a deterministic optimization problem are also obtained.
	\end{abstract}
	
	\textbf{Key words}: $G$-expectation, backward stochastic differential equations,  double mean reflections, non-Lipschitz coefficients, comparison theorem, optimization problem
	
	\textbf{MSC-classification}: 60G65, 60H10

\section{Introduction}
We firstly introduce the background in Subsection 1.1 and then indicates our contributions in Subsection 1.2, followed with the organization of the paper in Subsection 1.3.

\subsection{Background}
In 1997, El Karoui et al. \cite{E} systematically investigated the reflected backward stochastic differential equations (reflected BSDEs) of the following form
\begin{align}\label{1}
Y_{t}=\xi+\int_{t}^{T} f\left(s, Y_{s}, Z_{s}\right) d s-\int_{t}^{T} Z_{s} d W_{s}+A_T-A_t,
\end{align}
where the  nondecreasing process $A$ interprets the extra cost for keeping $Y$ above  the given constraint $\ell$, i.e., for any $t\in[0,T]$, $Y_t\geq \ell_t$. Moreover, this process $A$ should satisfy the Skorokhod condition to make the constraint hold in a minimum fashion.
 Then, Cvitani\'{c} and Karatzas \cite{CK} considered the BSDEs with double reflections, where the solution stays between two prescribed obstacles. It was shown that the solution to a reflected BSDE and to a doubly reflected BSDE corresponds to the value function of an optimal stopping problem and the value function of a Dynkin game, respectively.  

However, the constraints described above are made on the paths of the solution which are not valid when we consider the quantile hedging and related target problems with controlled loss, see \cite{B1}. Inspired by this fact, Briand, Elie and Hu \cite{BEH} introduced the so-called mean reflected BSDEs satisfying \eqref{1} subject to a weaker constraint on the distribution of the component $Y$ taking the following form
$$
{\mathbb{E}}\left[\ell\left(t, Y_t\right)\right] \geq 0, \quad \forall t \in[0, T],
$$
for a given loss function $\ell$, where $\mathbb{E}[\cdot]$ is the classical expectation. In order to get the corresponding existence and uniqueness results, they propose the notion of deterministic flat solution meaning that $A$ is required to be a deterministic nondecreasing process satisfying the condition of Skorokhod's type. 
Due to its wide applications in mathematical finance, such as superhedging of contingent claims under a running risk management constraint, the mean reflected BSDE has received considerable attention. Briand and Hibon \cite{BH} investigated the propagation of chaos for  mean reflected BSDEs. Moreover, Falkowski and Slomi\'{n}ski  \cite{FS2} considered the case of mean reflection with two constraints:
$$
{\mathbb{E}}\left[\ell\left(t, Y_t\right)\right]\in [l_t,r_t], \quad \forall t \in[0, T].
$$
Similar to \cite{BH}, Li and Ning \cite{LN} successfully approximated the doubly mean reflected BSDE by interacting particle systems. Recently, Li \cite{Li} further generalized the results in \cite{FS2} to a more general situation with two nonlinear reflecting boundaries, that is, we are given two different nonlinear loss functions. 
 The readers may refer to  \cite{BCGL,DE,HHL,HH,HMW,HT1,QW} and the references therein for a comprehensive overview of this theory.

It should be pointed out that  both the reflected and mean reflected BSDEs in the classical case cannot be applied to the pricing problem when there exists volatility uncertainty in the underlying financial markets. The $G$-expectation theory introduced by Peng \cite{P07a,P08a,P19} has become a very useful tool to deal with volatility ambiguity in finance as presented in \cite{EJ}. Therefore, it has been developed notably to include most of the classical results of probability theory and stochastic calculus, see for example \cite{DHP11,G,HWZ,STZ1}. In particular, Hu, Ji, Peng and Song \cite{HJPS1,HJb} investigated the BSDE theory under this $G$-framework {(Soner, Touzi and Zhang \cite{STZ} developed another formulation of fully nonlinear BSDE called 2BSDE, which shares many similarities with $G$-BSDEs). 
Moreover, the exploration of reflected $G$-BSDEs can be found in \cite{LP',LP,LS'}, which consider the case of an upper obstacle, the case of a lower obstacle and the case of two obstacles, respectively. In contrast to the pathwise constraint case, it is natural to investigate $G$-BSDEs with mean reflection. 
Liu and Wang \cite{LiuW} have taken the lead in making progress in this area. However, due to the special structure of $G$-BSDEs, we cannot construct a contraction mapping involving $Z$ since there is no appropriate estimate for this term. Therefore, in \cite{LiuW}, they only get the existence and uniqueness results for two special kinds of generators. Later, He \cite{HE} extends the results in \cite{LiuW} to a multi-dimensional and time-varying non-Lipschitz setting and it is worth noting that the general Lipschitz case can be included as its special case. What's more, Gu et al. \cite{GLX} established the existence result when the generators have  quadratic growth in $Z$. A natural question arises regarding whether doubly mean reflected BSDEs driven by $G$-Brownian motion can be well-defined. Besides, for both the doubly reflected and doubly mean reflected BSDEs in the classical expectation framework, the solution coincides with the value function of a game problem (see \cite{CK,DQS,GIOQ,FS2,Li}). It appears worthwhile to investigate if the solution to a doubly mean reflected $G$-BSDE can be regarded as the value function of certain optimization problem. Hence, it is the objective of this paper after establishing the well-posedness of doubly mean reflected $G$-BSDEs.

\subsection{Our contributions}
In this paper, we introduce the mean reflected $G$-BSDEs with two nonlinear reflecting boundaries in the following form
\begin{equation}\label{intro1}
    \begin{cases}
Y_t=\xi+\int_t^T f(s,Y_s,Z_s)ds+\int_t^T g(s,Y_s,Z_s)d\langle B\rangle_s-\int_t^T Z_s dB_s-(K_T-K_t)+(A_T-A_t), \\
\hE[L(t,Y_t)]\leq 0\leq \hE[R(t,Y_t)], \\
A_t=A^R_t-A^L_t \textrm{ and } \int_0^T \hE[R(t,Y_t)]dA_t^R=\int_0^T \hE[L(t,Y_t)]dA^L_t=0,
\end{cases}
\end{equation}
where $B$ denotes the $G$-Brownian motion, $\hE[\cdot]$ represents the $G$-expectation, $K$ is a non-increasing $G$-martingale, the generators $f,g$ and the loss functions $L,R$ satisfy certain regularity conditions (see Assumptions \ref{ass2}, \ref{assfg} and \ref{assfg2} below). The solution to this kind of constrained $G$-BSDE is a quadruple of processes $(Y,Z,K,A)$ with $A$ being the difference between two  nondecreasing deterministic functions $A^R$, $A^L$, each of which satisfy the Skorokhod condition. The existence and uniqueness of solution to doubly mean reflected $G$-BSDE \eqref{intro1} is established in Theorem \ref{thm3.10} under the same Lipschitz continuous assumption as made for the generators of the non-reflected case in \cite{HJPS1}, followed by the well-posedness of a unique solution to \eqref{intro1} with weaker regularity on the $Y$ term established in Theorem \ref{myw310}.

In order to prove Theorem \ref{thm3.10}, we first construct the solution  when the generators do not depend on $Y$ and $Z$ (see Proposition \ref{prop7}), drawing support from the backward Skorokhod problem with two nonlinear reflecting boundaries obtained in \cite{Li}. Due to the lack of appropriate estimate for the $Z$ term, the method of simultaneously applying contraction mappings to $Y$ and $Z$ has been proven ineffective. Based on the construction for the case of constant generators and the representation for doubly mean reflected $G$-BSDEs utilizing $G$-BSDEs, Theorem \ref{thm3.8} exhibits the well-posedness of doubly mean reflected $G$-BSDEs for some special case where the generator $f$ is  deterministically linearly dependent on $y$ and $g$ does not depend on $y$. As for the general Lipschitz case, employing the fix-point argument only for the $Y$ term, we could derive the local well-posedness result.  
The global situation is then derived through a backward iteration of the local ones. Besides,  motivated by Mao \cite{Mao} and He \cite{H}, we would like to further make efforts to relax the Lipschitz assumption of the generators. With the help of the representation mentioned above and using Picard iteration argument only for the $Y$ component, we could also establish the well-posedness result under a so-called $\beta$-order Mao's condition which is presented  in Theorem \ref{myw310}. It should be point out that the well-known Bihari's inequality and the a priori estimates of $G$-BSDEs obtained \cite{HJPS1} play an important role in the proof. 

Note that the Skorokhod conditions ensure the minimality of the force $A^R$ aiming to push the solution upwards and the force $A^L$ to pull the solution downwards, respectively. We explore this effect by a comparison theorem, denoted as Proposition \ref{prop11}. Besides, another kind of comparison theorem concerning the loss functions is established (see Proposition \ref{prop11'}). It is worth noting that, since the constraints are given in expectation but not pointwisely, the comparison property only holds for some special structure of the generators. Besides, we creatively develop a new connection between the expectation of the first component of the solution to the doubly mean reflected $G$-BSDE \eqref{intro1} and an appropriate optimization problem. Recall that the solution to a doubly mean reflected BSDE in the classical case coincides with the value of a ``$\sup_s \inf_t$" problem (see \cite{FS2,Li}). However, since the $G$-expectation can be represented as an upper expectation over a set of non-dominated probability measures, the optimization problems considered in this paper are of the ``$\sup_s \inf_t \sup_P$" type and the ``$\inf_t\sup_s \inf_P$" type, which are more complicated. Actually, in Theorem \ref{theorem3.6}, we show that the lower expectation and the upper expectation of $Y$, which is the solution to \eqref{intro1}, lie between the lower value and the upper value of an appropriate optimization problem. Especially, when the solution to the doubly mean reflected $G$-BSDE has no mean uncertainty, the values coincides with each other.  These findings may bring some inspiration to the stochastic control problems and financial issues in our future research.

\subsection{Organization of the paper}
The paper is organized as follows. In Section 2, we review some basic notations of $G$-expectation theory and some existing results of backward Skorokhod problem. Then, we study the well-posedness problem of mean reflected $G$-BSDEs with two constraints in Section 3. Section 4 is devoted to the properties of the solution including the comparison theorem and the connection with an optimization problem.

\section{Preliminaries}
In this paper, for a given set of parameters or functions $\alpha$, $\widetilde{C}(\alpha)$ will denote a positive constant only depending on these parameters or functions and may change from line to line.
\subsection{$G$-expectation theory}
Firstly, we review some basic notions and results of $G$-expectation together with $G$-stochastic calculus.  The readers may refer to  \cite{P07a,P08a,P19} for more details. For simplicity, we only consider the 1-dimensional $G$-Brownian motion. The results still hold for the multidimensional case.


	Let $\Omega_T=C_{0}([0,T];\mathbb{R})$, the space of
real-valued continuous functions with $\omega_0=0$, be endowed
with the supremum norm and 
let  $B$ be the canonical
process. Set
\[
L_{ip} (\Omega_T):=\{ \varphi(B_{t_{1}},...,B_{t_{n}}):  \ n\in\mathbb {N}, \ t_{1}
,\cdots, t_{n}\in\lbrack0,T], \ \varphi\in C_{b,Lip}(\mathbb{R}^{ n})\},
\]
where $C_{b,Lip}(\mathbb{R}^{ n})$ denotes the set of bounded Lipschitz functions on $\mathbb{R}^{n}$.

We fix a sublinear, continuous and monotone function  $G:\mathbb{R}\rightarrow\mathbb{R}$ defined by
\begin{displaymath}
G(a):=\frac{1}{2}(\bar{\sigma}^2a^+-\underline{\sigma}^2a^-),
\end{displaymath}
where $0\leq \underline{\sigma}^2<\bar{\sigma}^2$. The related $G$-expectation on $(\Omega,L_{ip}(\Omega_T))$ can be constructed in the following way. Assume that $\xi\in L_{ip}(\Omega_T)$ can be represented as
    \begin{displaymath}
    	\xi=\varphi(B_{{t_1}}, B_{t_2},\cdots,B_{t_n}).
\end{displaymath}
    Then, for $t\in[t_{k-1},t_k)$, $k=1,\cdots,n$,
\begin{displaymath}
	\hat{\mathbb{E}}_{t}[\varphi(B_{{t_1}}, B_{t_2},\cdots,B_{t_n})]=u_k(t, B_t;B_{t_1},\cdots,B_{t_{k-1}}),
\end{displaymath}
where, for any $k=1,\cdots,n$, $u_k(t,x;x_1,\cdots,x_{k-1})$ is a function of $(t,x)$ parameterized by $(x_1,\cdots,x_{k-1})$ such that it solves the following fully nonlinear PDE defined on $[t_{k-1},t_k)\times\mathbb{R}$:
\begin{displaymath}
	\partial_t u_k+G(\partial_x^2 u_k)=0
\end{displaymath}
with terminal conditions
\begin{displaymath}
	u_k(t_k,x;x_1,\cdots,x_{k-1})=u_{k+1}(t_k,x;x_1,\cdots,x_{k-1},x), \ k<n
\end{displaymath}
and $u_n(t_n,x;x_1,\cdots,x_{n-1})=\varphi(x_1,\cdots,x_{n-1},x)$. Hence, the $G$-expectation of $\xi$ is $\hat{\mathbb{E}}_0[\xi]$ and for simplicity, we always omit the subscript $0$. The triple $(\Omega,L_{ip}(\Omega_T),\hat{\mathbb{E}})$ is called the $G$-expectation space.


Define $\Vert\xi\Vert_{L_{G}^{p}}:=(\hat{\mathbb{E}}[|\xi|^{p}])^{1/p}$ for $\xi\in L_{ip}(\Omega_T)$ and $p\geq1$.   The completion of $L_{ip} (\Omega_T)$ under this norm  is denote by $L_{G}^{p}(\Omega)$. For all $t\in[0,T]$, $\hat{\mathbb{E}}_t[\cdot]$ is a continuous mapping on $L_{ip}(\Omega_T)$ w.r.t the norm $\|\cdot\|_{L_G^1}$. Hence, the conditional $G$-expectation $\mathbb{\hat{E}}_{t}[\cdot]$ can be
extended continuously to the completion $L_{G}^{1}(\Omega_T)$. Denis, Hu and Peng \cite{DHP11} prove that the $G$-expectation has the following representation.
\begin{theorem}[\cite{DHP11}]
	\label{the1.1}  There exists a weakly compact set
	$\mathcal{P}$ of probability
	measures on $(\Omega_T,\mathcal{B}(\Omega_T))$, such that
	\[
	\hat{\mathbb{E}}[\xi]=\sup_{\P\in\mathcal{P}}\E^{\P}[\xi] \text{ for all } \xi\in  {L}_{G}^{1}{(\Omega_T)}.
	\]
	$\mathcal{P}$ is called a set that represents $\hat{\mathbb{E}}$.
\end{theorem}

Let $\mathcal{P}$ be a weakly compact set that represents $\hat{\mathbb{E}}$.
A set $A\in\mathcal{B}(\Omega_T)$ is called polar if $V(A)=0$.  A
property holds $``quasi$-$surely"$ (q.s.) if it holds outside a
polar set. In this paper, we do not distinguish two random variables $X$ and $Y$ if $X=Y$, q.s.. The technical lemma below  will be used throughout Subsection 3.4 in this paper.

\begin{lemma}[Jensen's inequality \cite{BL}] \label{myw210}Let $\rho: \mathbb{R} \rightarrow \mathbb{R}$ be a continuous, non-decreasing and concave function. Then, for each $X \in$ $L_{G}^{1}\left(\Omega_{T}\right)$ such that $\rho(X)\in L_{G}^{1}\left(\Omega_{T}\right)$, the following inequality holds:
\[
\rho\Big(\hat{\mathbb{E}}[X]\Big) \geq \hat{\mathbb{E}}\Big[\rho(X)\Big].
\]
\end{lemma}






\begin{definition}
	\label{def2.6} Let $M_{G}^{0}(0,T)$ be the collection of processes in the
	following form: for a given partition $\{t_{0},\cdot\cdot\cdot,t_{N}\}=\pi
	_{T}$ of $[0,T]$,
	\[
	\eta_{t}(\omega)=\sum_{j=0}^{N-1}\xi_{j}(\omega)\mathbf{1}_{[t_{j},t_{j+1})}(t),
	\]
	where $\xi_{i}\in L_{ip}(\Omega_{t_{i}})$, $i=0,1,2,\cdot\cdot\cdot,N-1$. For each
	$p\geq1$ and $\eta\in M_G^0(0,T)$, let $\|\eta\|_{H_G^p}:=\{\hat{\mathbb{E}}[(\int_0^T|\eta_s|^2ds)^{p/2}]\}^{1/p}$, $\Vert\eta\Vert_{M_{G}^{p}}:=(\mathbb{\hat{E}}[\int_{0}^{T}|\eta_{s}|^{p}ds])^{1/p}$ and denote by $H_G^p(0,T)$,  $M_{G}^{p}(0,T)$ the completion
	of $M_{G}^{0}(0,T)$ under the norm $\|\cdot\|_{H_G^p}$, $\|\cdot\|_{M_G^p}$, respectively.
\end{definition}

 We denote by $\langle B\rangle$ the quadratic variation process of the $G$-Brownian motion $B$. For two processes $ \xi\in M_{G}^{1}(0,T)$ and $ \eta\in M_{G}^{2}(0,T)$,
the $G$-It\^{o} integrals $(\int^{t}_0\xi_sd\langle
B\rangle_s)_{0\leq t\leq T}$ and $(\int^{t}_0\eta_sdB_s)_{0\leq t\leq T}$ are well defined, see  Li and Peng \cite{lp} and Peng \cite{P19}. The following proposition can be regarded as the Burkholder--Davis--Gundy inequality under $G$-expectation framework
\begin{proposition}[\cite{P19}]\label{BDG}
	If $\eta\in H_G^{\alpha}(0,T)$ with $\alpha\geq 1$ and $p\in(0,\alpha]$, then we have
	\begin{displaymath}
	\underline{\sigma}^p c_p\hat{\mathbb{E}}_t[(\int_t^T |\eta_s|^2ds)^{p/2}]\leq
	\hat{\mathbb{E}}_t[\sup_{u\in[t,T]}|\int_t^u\eta_s dB_s|^p]\leq
	\bar{\sigma}^p C_p\hat{\mathbb{E}}_t[(\int_t^T |\eta_s|^2ds)^{p/2}],
	\end{displaymath}
	where $0<c_p<C_p<\infty$ are constants depending on $p, T$.
\end{proposition}

Similar to the linear stochastic analysis, in the $G$-expectation framework, we can also define martingales and obtain some related properties. The following Doob's type estimate is from \cite{STZ1, Song}, and extended in \cite{LS}. More details about this estimate could be found in \cite{P19}.
\begin{theorem}[\cite{LS}]\label{myw204} Let $1\leq\alpha<\beta .$ Then, for all $\xi \in L_{G}^{\beta}\left(\Omega_{T}\right),$ we can find a constant $\widetilde{C}(\alpha, \beta)>0$ such that
\[
\hat{\mathbb{E}}\Big[\sup _{t \in[0, T]} \hat{\mathbb{E}}_{t}\left[|\xi|^{\alpha}\right]\Big] \leq \widetilde{C}(\alpha, \beta)\hat{\mathbb{E}}\left[|\xi|^{\beta}\right]^{\frac{\alpha}{ \beta}}.
\]
\end{theorem}

Let $S_G^0(0,T)=\{h(t,B_{t_1\wedge t}, \ldots,B_{t_n\wedge t}):t_1,\ldots,t_n\in[0,T],h\in C_{b,Lip}(\mathbb{R}^{n+1})\}$. For $p\geq 1$ and $\eta\in S_G^0(0,T)$, set $\|\eta\|_{S_G^p}=\{\hat{\mathbb{E}}[\sup_{t\in[0,T]}|\eta_t|^p]\}^{1/p}$. Denote by $S_G^p(0,T)$ the completion of $S_G^0(0,T)$ under the norm $\|\cdot\|_{S_G^p}$.  We have the following continuity property for any $Y\in S_G^p(0,T)$ with $p>1$.

\begin{proposition}[\cite{LPS}]\label{the3.7}
For $Y\in S_G^p(0,T)$ with $p\geq1$, we have, by setting $Y_s:=Y_T$ for $s>T$,
\begin{displaymath}
F(Y):=\limsup_{\varepsilon\rightarrow0}(\hat{\mathbb{E}}[\sup_{t\in[0,T]}\sup_{s\in[t,t+\varepsilon]}|Y_t-Y_s|^p])^\frac{1}{p}=0.
\end{displaymath}
\end{proposition}


\subsection{Backward Skorokhod problem with two nonlinear reflecting boundaries}

In this subsection, we will recall the backward Skorokhod problem studied in \cite{Li}, which is important for the construction of solutions to $G$-BSDEs with double mean reflections. First, we introduce the following notations.

\begin{itemize}
\item $C[0,T]$: the set of continuous functions from $[0,T]$ to $\mathbb{R}$.
\item $BV[0,T]$: the set of functions in $C[0,T]$ starting from the origin with bounded variation on $[0,T]$.
\item $I[0,T]$: the set of functions in $C[0,T]$ starting from the origin which is nondecreasing.
\end{itemize}

\begin{definition}\label{def}
Let $s\in C[0,T]$, $a\in \mathbb{R}$ and $l,r:[0,T]\times \mathbb{R}\rightarrow \mathbb{R}$ be two functions such that $l\leq r$ and $l(T,a)\leq 0\leq r(T,a)$. A pair of functions $(x,k)\in C[0,T]\times BV[0,T]$ is called a solution of the backward Skorokhod problem for $s$ with nonlinear constraints $l,r$ ($(x,k)=\mathbb{BSP}_l^r(s,a)$ for short) if 
\begin{itemize}
\item[(i)] $x_t=a+s_T-s_t+k_T-k_t$;
\item[(ii)] $l(t,x_t)\leq 0\leq r(t,x_t)$, $t\in[0,T]$;
\item[(iii)] $k_{0}=0$ and $k$ has the decomposition $k=k^r-k^l$, where $k^r,k^l\in I[0,T]$ satisfy 
\begin{align}\label{iii}
\int_0^T l(s,x_s)dk^l_s=0, \ \int_0^T r(s,x_s)dk^r_s=0.
\end{align}
\end{itemize}
\end{definition}

In order to solve the backward Skorokhod problem, we make the following assumptions on the reflecting boundary functions.
\begin{assumption}\label{ass1}
The functions $l,r:[0,T]\times \mathbb{R}\rightarrow \mathbb{R}$ satisfy the following conditions:
\begin{itemize}
\item[(i)] For each fixed $x\in\mathbb{R}$, $l(\cdot,x),r(\cdot,x)\in C[0,T]$;
\item[(ii)] For any fixed $t\in[0,T]$, $l(t,\cdot)$, $r(t,\cdot)$ are strictly increasing;
\item[(iii)] There exists two positive constants $0<c<C<\infty$, such that for any $t\in [0,T]$ and $x,y\in \mathbb{R}$,
\begin{align*}
&c|x-y|\leq |l(t,x)-l(t,y)|\leq C|x-y|,\\
&c|x-y|\leq |r(t,x)-r(t,y)|\leq C|x-y|.
\end{align*}
\item[(iv)] $\inf_{(t,x)\in[0,T]\times\mathbb{R}}(r(t,x)-l(t,x))>0$.
\end{itemize}
\end{assumption}

\begin{theorem}[\cite{Li}]\label{BSP}
Let Assumption \ref{ass1} hold. For any given $s\in C[0,T]$ and $a\in \mathbb{R}$ with $l(T,a)\leq 0\leq r(T,a)$, there exists a unique solution to the backward Skorokhod problem $(x,k)=\mathbb{BSP}_l^r(s,a)$.
\end{theorem}

The following theorem gives the continuous dependence of the solution to the backward Skorokhod problem with respect to the terminal value, the input function and the reflecting boundary functions.

\begin{theorem}[\cite{Li}]\label{BSP'}
Given $a^i\in\mathbb{R}$, $s^i\in C[0,T]$, $l^i,r^i$ satisfying Assumption \ref{ass1} and $l^i(T,a^i)\leq 0\leq r^i(T,a^i)$, $i=1,2$,  let $(x^i,k^i)$ be the solution to the backward Skorokhod problem $\mathbb{BSP}_{l^i}^{r^i}(s^i,a^i)$. Then, we have
\begin{equation}\label{diffK}
\sup_{t\in[0,T]}|k^1_t-k^2_t|
\leq  2\frac{C}{c}|a^1-a^2|+4\frac{C}{c}\sup_{t\in[0,T]}|s^1_t-s^2_t|+\frac{2}{c}(\bar{L}_T\vee\bar{R}_T),
\end{equation}
where
\begin{align*}
&\bar{L}_T=\sup_{(t,x)\in[0,T]\times \mathbb{R}}|{l}^1(t,x)-{l}^2(t,x)|,\\
&\bar{R}_T=\sup_{(t,x)\in[0,T]\times \mathbb{R}}|{r}^1(t,x)-{r}^2(t,x)|.
\end{align*}
\end{theorem}

In what follows, we would like to present a useful estimate which  will play a crucial role in the Picard  iteration argument in Subsection 3.4.
\begin{theorem}\label{AA209}
Suppose that $(x,k)$ is the solution to the backward Skorokhod problem $\mathbb{BSP}_{l}^{r}(s,a)$. Then there exists a constant $\widetilde{C}(C,c,T,r,l)>0$ such that for any $t\in[0,T]$,
\begin{align*}
|k_T-k_t|\leq \widetilde{C}(C,c,T,r,l)\big(|a+s_T-s_t|+1\big) .
\end{align*}
\end{theorem}
\begin{proof}
Recalling the proof of Theorem 3.9 in \cite{Li}, if $(\widetilde{x},\widetilde{k})$ is the solution to the Skorokhod problem $\mathbb{SP}_{\widetilde{l}}^{\widetilde{r}}(\widetilde{s})$, then $(x,k)$ is the solution to $\mathbb{BSP}_{l}^{r}(s,a)$, where  for any $t\in[0,T]$,
\begin{align*}
&x_t=\widetilde{x}_{T-t},\ k_t=\widetilde{k}_T-\widetilde{k}_{T-t},\\
&\widetilde{s}_t:=a+s_T-s_{T-t},\  \widetilde{l}(t, x):=l(T-t, x), \ \widetilde{r}(t, x):=r(T-t, x).
\end{align*}
 Proceeding identically as the proof of Theprem 2.19 in \cite{Li2}, we could get
\begin{align*}
\Psi_t\leq \widetilde{k}_t \leq \Phi_t,
\end{align*}
where $\Psi_t,\Phi_t$ satisfy $\widetilde{l}(t,\widetilde{s}_t+\Psi_t)=0,\widetilde{r}(t,\widetilde{s}_t+\Phi_t)=0$ respectively, which further implies
\begin{align*}
|\widetilde{k}_t|\leq|\Psi_t|+| \Phi_t|.
\end{align*}
It follows from  the bi-Lipschitz continuity of $l,r$ in Assumption \ref{ass1} (iii) that
\begin{align*}
|\Psi_t|+| \Phi_t| &\leq \frac{1}{c}\Big(|\widetilde{l}(t,\widetilde{s}_t)|+|\widetilde{r}(t,\widetilde{s}_t)|\Big)\leq \frac{C}{c}\big(|\widetilde{l}(t,0)|+|\widetilde{r}(t,0)|+2|\widetilde{s}_t|\big)\\
&\leq \frac{C}{c}\big(\sup_{t\in[0,T]}|{l}(t,0)|+\sup_{t\in[0,T]}|{r}(t,0)|+2|\widetilde{s}_t|\big).
\end{align*}
In view of (i) in Assumption \ref{ass1}, we have $\sup_{t\in[0,T]}|{l}(t,0)|+\sup_{t\in[0,T]}|{r}(t,0)|\leq \widetilde{C}(T,r,l)$. 
Thus 
\begin{align*}
|k_T-k_t|=|\widetilde{k}_{T-t}|\leq \widetilde{C}(C,c,T,r,l) \big(|\widetilde{s}_{T-t}|+1\big)= \widetilde{C}(C,c,T,r,l) \big(|a+s_T-s_t|+1\big),
\end{align*}
which ends the proof.
\end{proof}


\section{$G$-BSDEs with double mean reflections}

\subsection{The case of constant coefficients}

In this subsection, we first study the well-posedness of BSDEs with double mean reflections driven by $G$-Brownian motion of the following form
\begin{equation}\label{nonlinearnoyz}
\begin{cases}
Y_t=\xi+\int_t^T C_sds+\int_t^T D_s d\langle B\rangle_s-\int_t^T Z_s dB_s-(K_T-K_t)+(A_T-A_t), \\
\hE[L(t,Y_t)]\leq 0\leq \hE[R(t,Y_t)], \\
A_t=A^R_t-A^L_t, \ A^R,A^L\in I[0,T] \textrm{ and } \int_0^T \hE[R(t,Y_t)]dA_t^R=\int_0^T \hE[L(t,Y_t)]dA^L_t=0,
\end{cases}
\end{equation}
where $C,D$ are two given processes in $M_G^\beta(0,T)$ with $\beta>1$ and $L,R$ are two running loss functions. We make the following assumption on $L,R$.

\begin{assumption}\label{ass2}
The running loss functions $L,R:\Omega_T\times [0,T]\times\mathbb{R}\rightarrow \mathbb{R}$ satisfy the following conditions:
\begin{itemize}
\item[(1)] $(t,x)\rightarrow L(\omega,t,x)$, $(t,x)\rightarrow R(\omega,t,x)$ are uniformly continuous, uniform in $\omega$;
\item[(2)]  for any fixed $(t,x)\in [0,T]\times \mathbb{R}$, $L(t,x),R(t,x)\in L_G^1(\Omega_T)$;
\item[(3)] for any fixed $(\omega,t)\in \Omega_T\times [0,T]$, $L(\omega,t,\cdot),R(\omega,t,\cdot)$ are strictly increasing and there exists two constants $0<c<C<\infty$ such that for any $x,y\in \mathbb{R}$,
\begin{align*}
&c|x-y|\leq |L(\omega,t,x)-L(\omega,t,y)|\leq C|x-y|,\\
&c|x-y|\leq |R(\omega,t,x)-R(\omega,t,y)|\leq C|x-y|;
\end{align*}
\item[(4)] for any fixed $(\omega,t)\in \Omega_T\times [0,T]$, there exists a constant $M\geq 0$ such that
\begin{align*}
|L(t,y)|\leq M(1+|y|), \ |R(t,y)|\leq M(1+|y|);
\end{align*}
\item[(5)] $\inf_{(\omega,t,x)\in \Omega_T\times [0,T]\times\mathbb{R}} (R(\omega,t,x)-L(\omega,t,x))>0$.
\end{itemize}
\end{assumption}

\begin{remark}
    Conditions (1)-(4) in Assumption \ref{ass2} are the common requirements for the loss functions under $G$-framework (see Assumptions ($H_l$) and ($H'_l$) in \cite{LiuW}). The solution to doubly mean reflected $G$-BSDE is closely related to the solution to a certain backward Skorokhod problem. The effect of Condition (5) is to make sure the well-posedness of this backward Skorokhod problem since it requires the completely separated condition for the reflecting boundary functions (see Theorem \ref{BSP} and  Condition (iv) in Assumption \ref{ass1}).
\end{remark}

 Motivated by the construction of solutions to $G$-BSDEs with single mean reflecting boundary investigated in \cite{LiuW} and to BSDEs with double mean reflections in the classical case studied in \cite{Li}, for any fixed $S\in {S}^1_G(0,T)$, we define the following two maps $r,l:[0,T]\times\mathbb{R}\rightarrow\mathbb{R}$ as
\begin{equation}\label{operator}
\begin{split}
r^S(t,x):=\hE[R(t,x+S_t-\hE[S_t])],\  l^S(t,x):=\hE[L(t,x+S_t-\hE[S_t])].
\end{split}
\end{equation}
For simplicity, we always omit the superscript $S$. First, by Proposition 3.3 (i) in \cite{LiuW}, for any $(t,x)\in[0,T]\times\mathbb{R}$ and $S\in S_G^1(0,T)$, we have
$$
R(t,x+S_t-\hE[S_t]),L(t,x+S_t-\hE[S_t])\in L_G^1(\Omega_T).
$$
Consequently, $r(t,x)$, $l(t,x)$ are well-defined. Then, we show that $r,l$ satisfy Assumption \ref{ass1}, which ensures that they can be considered as reflecting boundary functions for some backward Skorokhod problem.

\begin{lemma}\label{proofofass1}
Under Assumption \ref{ass2}, for any $S\in {S}^1_G(0,T)$, $l,r$ defined by \eqref{operator} satisfy Assumption \ref{ass1}.
\end{lemma}

\begin{proof}
We first show that Assumption \ref{ass1} (iv) holds. Indeed, by the sublinear property of $\hE[\cdot]$ and Assumption \ref{ass2} (5), we have
\begin{align*}
r(t,x)-l(t,x)&\geq -\hE[-(R(t,x+S_t-\hE[S_t])-L(t,x+S_t-\hE[S_t]))]\\
&\geq \inf_{\omega,t,x}(R(\omega,t,x)-L(\omega,t,x))>0.
\end{align*}

In the following, we only show that $r$ satisfies Assumption \ref{ass1} (i)-(iii). The proof of continuity property and strict increasing property is similar with the one for Proposition 3.3 in \cite{LiuW}. For readers' convenience, we give a short proof here. By Assumption \ref{ass2} (1), for any $\varepsilon>0$, there exists a constant $\delta$ only depending on $\varepsilon$, such that for any $|t-s|\leq \delta$, we have
$$
|R(t,x)-R(s,x)|\leq \varepsilon.
$$
Simple calculation yields that
\begin{align*}
|r(t,x)-r(s,x)|&\leq\hE[|R(t,x+S_t-\hE[S_t])-R(s,x+S_s-\hE[S_s])|]\\
&\leq \hE[|R(t,x+S_t-\hE[S_t])-R(s,x+S_t-\hE[S_t])|]+2C\hE[|S_t-S_s|].
\end{align*}
By the continuity property of $S$ (see Proposition \ref{the3.7}), $r$ satisfies Assumption \ref{ass1} (i).

Now, given $x>y$, we have $R(t,x+S_t-\hE[S_t])>R(t,y+S_t-\hE[S_t])$, q.s. Noting that $R(t,y+S_t-\hE[S_t])\in L_G^1(\Omega_T)$, the representation theorem of $G$-expectation implies that there exists some $\P\in\mathcal{P}$, such that
\begin{align*}
\hE[R(t,y+S_t-\hE[S_t])]=\E^\P[R(t,y+S_t-\hE[S_t])].
\end{align*}
It follows that
\begin{align*}
r(t,x)&=\hE[R(t,x+S_t-\hE[S_t])]\geq\E^\P[R(t,x+S_t-\hE[S_t])]\\
&>\E^\P[R(t,y+S_t-\hE[S_t])]=\hE[R(t,y+S_t-\hE[S_t])]=r(t,y),
\end{align*}
which implies that $r$ is strictly increasing. It remains to prove the bi-Lipschitz property of $r$.
 For any $x,y\in \mathbb{R}$ with $x<y$, we have
\begin{align*}
&r(t,y)-r(t,x)\leq \hE[R(t,y+S_t-\hE[S_t])-R(t,x+S_t-\hE[S_t])]\leq C(y-x),\\
&r(t,y)-r(t,x)\geq -\hE[-(R(t,y+S_t-\hE[S_t])-R(t,x+S_t-\hE[S_t]))]\geq c(y-x).
\end{align*}
Hence, $r$ satisfies Assumption \ref{ass1} (iii).  The proof is complete.
\end{proof}

\begin{proposition}\label{prop7}
Suppose that $L,R$ satisfy Assumption \ref{ass2}. Given $\xi\in L_G^\beta(\Omega_T)$ and $C,D\in M_G^\beta(0,T)$ with
$$
\hE[L(T,\xi)]\leq 0\leq \hE[R(T,\xi)],
$$
where $\beta>1$, for any $1<\alpha<\beta$, the $G$-BSDE with double mean reflections \eqref{nonlinearnoyz} has a unique solution $(Y,Z,K,A)\in \mathfrak{S}^\alpha(0,T)\times BV[0,T]$, where $\mathfrak{S}^\alpha(0,T)$ is the collection of processes $(Y',Z',K')$ such that $Y'\in S_G^\alpha(0,T)$, $Z'\in H_G^\alpha(0,T)$ and $K'\in S_G^\alpha(0,T)$ is a nonincreasing $G$-martingale.
\end{proposition}

\begin{proof}
For simplicity, we only prove for the case that $D=0$. We first prove the uniqueness. 
Suppose that $(Y^i,Z^i,K^i,A^i)$ are solutions to \eqref{nonlinearnoyz}, $i=1,2$. Set $X_t=\hE_t[\xi+\int_t^T C_sds]$. Noting that $A^i$ is deterministic and
$$
Y_t^i+K^i_T-K^i_t=\xi+\int_t^T C_s ds-\int_t^T Z^i_s dB_s+A^i_T-A^i_t,
$$
taking conditional expectations on both sides yields that $Y_t^i=X_t+A^i_T-A^i_t$. Suppose that there exists $t_1<T$ such that
\begin{displaymath}
A^1_T-A^1_{t_1}>A^2_T-A^2_{t_1}.
\end{displaymath}
Set
\begin{displaymath}
t_2=\inf\{t\geq t_1: A_T^1-A^1_t=A_T^2-A^2_t\}.
\end{displaymath}
The definition of $t_2$ implies that
\begin{displaymath}
A_T^1-A^1_t> A^2_T-A^2_t, \ t_1\leq t<t_2.
\end{displaymath}
By a similar analysis as the strictly increasing property for $r$ in the proof of Lemma \ref{proofofass1}, we have for any $t_1\leq t<t_2$
\begin{align*}
&\hE[R(t,X_t+A_T^1-A_t^1)]>\hE[R(t,X_t+A_T^2-A_t^2)]\geq 0, \\
&\hE[L(t,X_t+A_T^2-A_t^2)]<\hE[L(t,X_t+A_T^1-A_t^1)]\leq 0.
\end{align*}
Applying the Skorokhod condition
\begin{displaymath}
\int_{t_1}^{t_2}\hE[R(t,X_t+A_T^1-A_t^1)]dA^{1,R}_t=\int_{t_1}^{t_2}\hE[L(t,X_t+A_T^2-A_t^2)]dA^{2,L}_t=0,
\end{displaymath}
we obtain that $dA^{1,R}_t=dA^{2,L}_t=0$ on the interval $[t_1,t_2]$. Recalling that $A^{1,L}$ and $A^{2,R}$ are nondecreasing, it is easy to check that
\begin{align*}
A_T^1-A_{t_2}^1=&A^1_T-(A^{1,R}_{t_2}-A^{1,L}_{t_2})=A^{1}_T-A^{1,R}_{t_1}+A^{1,L}_{t_2}\geq A^1_T-A^{1,R}_{t_1}+A^{1,L}_{t_1}\\
=&A_T^1-A^1_{t_1}>A^2_T-A^2_{t_1}=A_T^2-(A^{2,R}_{t_1}-A^{2,L}_{t_1})\\
=&A^2_T+A^{2,L}_{t_2}-A^{2,R}_{t_1}\geq A^2_T+A^{2,L}_{t_2}-A^{2,R}_{t_2}= A^2_T-A^2_{t_2},
\end{align*}
which contradicts the definition of $t_2$. Therefore, we have $A^1\equiv A^2=:A$. Then, $(Y^i,Z^i,K^i)$ can be regarded as the solution to the following $G$-BSDE, $i=1,2$
$$
Y^i_t=\xi+\int_t^T C_s ds-\int_t^TZ^i_sd B_s-(K_T^i-K^i_t)+(A_T-A_t).
$$
By the uniqueness of solutions to $G$-BSDEs, we have $Y^1\equiv Y^2$, $Z^1\equiv Z^2$ and $K^1\equiv K^2$.

Now we prove the existence. For any fixed $\alpha\in(1,\beta)$, let $(\widetilde{Y},Z,K)\in \mathfrak{S}^\alpha(0,T)$ be the solution to the $G$-BSDE with terminal value $\xi$ and constant coefficient $C$, i.e.,
\begin{align*}
\widetilde{Y}_t=\xi+\int_t^T C_s ds+\int_t^T Z_s dB_s-(K_T-K_t).
\end{align*}
  For any $t\in[0,T]$,  set
\begin{align*}
s_t=\hE[\xi+\int_0^T C_s ds]-\hE[\xi+\int_t^T C_sds], \ a=\hE[\xi].
\end{align*}
It is easy to check that $s\in C[0,T]$ and $s_t=\hE[\widetilde{Y}_0]-\hE[\widetilde{Y}_t]$.
 For any $(t,x)\in[0,T]\times \mathbb{R}$, we define
\begin{align*}
l(t,x):=\hE[L(t,\widetilde{Y}_t-\hE[\widetilde{Y}_t]+x)], \ r(t,x):=\hE[R(t,\widetilde{Y}_t-\hE[\widetilde{Y}_t]+x)].
\end{align*}
By Lemma \ref{proofofass1}, $l,r$ satisfy Assumption \ref{ass1} and
\begin{align*}
l(T,a)=\hE[L(T,\xi)]\leq 0\leq \hE[R(T,\xi)]=r(T,a).
\end{align*}
Let $(x,A)$ be the unique solution to the backward Skorokhod problem $\mathbb{BSP}_l^r(s,a)$. Now, we set
\begin{align*}
Y_t=\widetilde{Y}_t+(A_T-A_t).
\end{align*}
We claim that $(Y,Z,K,A)$ is the solution to \eqref{nonlinearnoyz}. It suffices to prove the Skorokhod conditions hold. In fact, simple calculation yields that
\begin{align*}
\hE[L(t,Y_t)]=&\hE[L(t,\widetilde{Y}_t+A_T-A_t)]\\
=&\hE[L(t,\widetilde{Y}_t+x_t-a-(s_T-s_t))]\\
=&\hE[L(t,\widetilde{Y}_t-\hE[\widetilde{Y}_t]+x_t)]=l(t,x_t).
\end{align*}
Similarly, we have $\hE[R(t,Y_t)]=r(t,x_t)$. Recalling that $(x,A)=\mathbb{BSP}_l^r(s,a)$, it follows that
\begin{align*}
&\hE[L(t,Y_t)]=l(t,x_t)\leq 0\leq r(t,x_t)=\hE[R(t,Y_t)],\\
&\int_0^T \hE[L(t,Y_t)]dA^L_t=\int_0^T l(t,x_t)dA^L_t=0,\\
&\int_0^T \hE[R(t,Y_t)]dA^L_t=\int_0^T r(t,x_t)dA^R_t=0.
\end{align*}
The proof is complete.
\end{proof}

 We establish the following the estimate for the bounded variation component of solutions to $G$-BSDEs with double mean reflections.
\begin{proposition}\label{proposition}
Suppose that $L,R$ satisfy Assumption \ref{ass2}. Given $\xi^i\in L_G^\beta(\Omega_T)$ and $C^i,D^i\in M_G^\beta(0,T)$ with
$$
\hE[L(T,\xi^i)]\leq 0\leq \hE[R(T,\xi^i)],
$$
where $\beta>1$, $i=1,2$, let $(Y^i,Z^i,K^i,A^i)$ be the solution to doubly mean reflected $G$-BSDE with terminal value $\xi^i$, loss functions $L,R$, constant coefficients $C^i,D^i$. Then, there exists a constant $ \widetilde{C}(c,C,G)$ such that
\begin{align*}
\sup_{t\in[0,T]}|A^1_t-A^2_t|\leq \widetilde{C}(c,C,G)(\hE[|\xi^1-\xi^2|]+\hE[\int_0^T|C^1_s-C^2_s|ds]+\hE[\int_0^T|D^1_s-D^2_s|ds]).
\end{align*}
\end{proposition}

\begin{proof}
By the proof of Proposition \ref{prop7} and the estimate for solutions to backward Skorokhod problem (see Theorem \ref{BSP'}), we have
\begin{align*}
\sup_{t\in[0,T]}|A^1_t-A^2_t|\leq 2\frac{C}{c}|a^1-a^2|+4\frac{C}{c}\sup_{t\in[0,T]}|s^1_t-s^2_t|+2\frac{C}{c}\sup_{(t,x)\in[0,T]\times \mathbb{R}}(|\hat{L}(t,x)|\vee|\hat{R}(t,x)|),
\end{align*}
where for $i=1,2$, $a^i=\hE[\xi^i]$ and
\begin{align*}
&s^i_t=\hE[\xi^i+\int_0^T C^i_sds+\int_0^T D^i_sd\langle B\rangle_s]-\hE[\xi^i+\int_t^T C^i_sds+\int_t^T D^i_sd\langle B\rangle_s],\\
&\hat{L}(t,x):=\hE[L(t,\widetilde{Y}^1_t-\hE[\widetilde{Y}^1_t]+x)]-\hE[L(t,\widetilde{Y}^2_t-\hE[\widetilde{Y}^2_t]+x)],\\
&\hat{R}(t,x):=\hE[R(t,\widetilde{Y}^1_t-\hE[\widetilde{Y}^1_t]+x)]-\hE[R(t,\widetilde{Y}^2_t-\hE[\widetilde{Y}^2_t]+x)],\\
&\widetilde{Y}^i_t=\xi^i+\int_t^T C^i_sds+\int_t^T D^i_sd\langle B\rangle_s+\int_t^T Z^i_sdB_s-(K^i_T-K^i_t).
\end{align*}
Applying the Lipschitz continuity of $L,R$, we obtain the desired result. 
\end{proof}

\subsection{The case of special nonlinear coefficients}

In this subsection, we consider the case where the coefficient $f$ is deterministically linearly dependent on $y$ and $g$ does not depend on $y$, which coincides with Case I in \cite{LiuW}. More precisely, the $G$-BSDE with double mean reflections takes the following form
\begin{equation}\label{nonlinearyz11}
\begin{cases}
Y_t=\xi+\int_t^T f(s,Y_s,Z_s)ds+\int_t^T g(s,Z_s)d\langle B\rangle_s-\int_t^T Z_s dB_s-(K_T-K_t)+(A_T-A_t), \\
\hE[L(t,Y_t)]\leq 0\leq \hE[R(t,Y_t)], \\
A_t=A^R_t-A^L_t, \ A^R,A^L\in I[0,T] \textrm{ and } \int_0^T \hE[R(t,Y_t)]dA_t^R=\int_0^T \hE[L(t,Y_t)]dA^L_t=0,
\end{cases}
\end{equation}
where $f,g$ satisfy the following assumptions.

\begin{assumption}\label{assfg1}
\begin{itemize}
\item[1.] $f:[0,T]\times \Omega_T\times\mathbb{R}\times\mathbb{R}\rightarrow \mathbb{R}$ is of the following form
\begin{align*}
f(t,y,z)=\gamma_t y+f'(t,z),
\end{align*}
\end{itemize}
where $\gamma$ is a deterministic and bounded Borel measurable function;
\item[2.] for each $z$, $f'(\cdot,z),g(\cdot,z)\in M_G^\beta(0,T)$ for some $\beta>1$;
\item[3.] there exists some constant $\kappa>0$, such that for any $t\in[0,T]$ and $z,z'\in\mathbb{R}$,
\begin{align*}
|f'(t,z)-f'(t,z')|+|g(t,z)-g(t,z')|\leq \kappa|z-z'|.
\end{align*}
\end{assumption}

\begin{theorem}\label{thm3.8}
Suppose that Assumption \ref{ass2} and \ref{assfg1} hold. Let $\xi\in L_G^\beta(\Omega_T)$ be such that
$\hE[L(T,\xi)]\leq 0\leq \hE[R(t,\xi)]$. Then, for each $1< \alpha<\beta$, the $G$-BSDE with double mean reflections \eqref{nonlinearyz11} has a unique solution $(Y,Z,K,A)\in \mathfrak{S}^\alpha_G\times BV[0,T]$. Moreover, we have $Y\in M_{G}^{\beta}\left(0, T \right)$. 
\end{theorem}

\begin{proof}
For simplicity, we only consider the case that $g=0$. The proof is analogous to the one for Theorem 3.8 in \cite{LiuW}. For readers' convenience, we give a short proof here. Set $a_t:=\int_0^t \gamma_s ds$ for each $t\in[0,T]$ and define
\begin{align*}
Y^a_t:=e^{a_t}Y_t,\ Z^a_t:=e^{a_t}Z_t, \ K^a_t:=\int_0^t e^{a_s}dK_s, \ A^a_t:=\int_0^t e^{a_s}dA_s.
\end{align*}
It is easy to check that $(Y,Z,K,A)$ is the solution to \eqref{nonlinearyz11} if and only if $(Y^a,Z^a,K^a,A^a)$ is the solution to doubly mean reflected $G$-BSDE with terminal value $\xi^a$, loss functions $L^a,R^a$ and coefficient $f^a$, where
\begin{align*}
\xi^a:=e^{a_T}\xi, \ L^a(t,x):=L(t,e^{-a_t}x), \ R^a(t,x)=:R(t,e^{-a_t}x), \ f^a(t,z):=e^{a_t}f'(t,e^{-a_t}z).
\end{align*}
Therefore, without loss of generality, we assume that the coefficient $f$ does not depend on $y$.

We first prove existence. Consider the $G$-BSDE with terminal value $\xi$, coefficient $f$ on time interval $[0,T]$. By Theorem \ref{mywA1}, for any $1<\alpha<\beta$, this $G$-BSDE admits a unique solution $(\bar{Y},\bar{Z},\bar{K})\in \mathfrak{S}_G^\alpha(0,T)$. Then, for any $1<\alpha'\leq 2\wedge \alpha$, we have $\bar{Z}\in H^{\alpha'}_G(0,T)\subset M_G^{\alpha'}(0,T)$. Applying Assumption \ref{assfg1} and Theorem 4.7 in \cite{HWZ} imply that $f(\cdot,\bar{Z}_\cdot)\in M_G^{\alpha'}(0,T)$. Now, let us consider the following $G$-BSDE with double mean reflections
\begin{displaymath}
\begin{cases}
\widetilde{Y}_t=\xi+\int_t^T f(s,\bar{Z}_s)ds-\int_t^T \widetilde{Z}_s dB_s-(\widetilde{K}_T-\widetilde{K}_t)+(\widetilde{A}_T-\widetilde{A}_t), \\
\hE[L(t,\widetilde{Y}_t)]\leq 0\leq \hE[R(t,\widetilde{Y}_t)], \\
\widetilde{A}_t=\widetilde{A}^R_t-\widetilde{A}^L_t, \  \widetilde{A}^R,\widetilde{A}^L\in I[0,T] \textrm{ and } \int_0^T \hE[R(t,\widetilde{Y}_t)]d\widetilde{A}_t^R=\int_0^T \hE[L(t,\widetilde{Y}_t)]d\widetilde{A}^L_t=0.
\end{cases}
\end{displaymath}
By Proposition \ref{prop7}, for any $1<\alpha''<\alpha'$, the above doubly mean reflected $G$-BSDE has a unique solution $(\widetilde{Y},\widetilde{Z},\widetilde{K},\widetilde{A})\in \mathfrak{S}^{\alpha''}_G(0,T)\times BV[0,T]$. On the other hand, both $(\bar{Y},\bar{Z},\bar{K})$ and $(\widetilde{Y}-(\widetilde{A}_T-\widetilde{A}),\widetilde{Z},\widetilde{K})$ can be seen as the solution to the following $G$-BSDE
\begin{align*}
\hat{Y}_t=\xi+\int_t^T f(s,\bar{Z}_s)ds-\int_t^T\hat{Z}_s dB_s-(\hat{K}_T-\hat{K}_t).
\end{align*}
By the uniqueness result of $G$-BSDEs, we deduce that
\begin{align*}
(\bar{Y}_t,\bar{Z}_t,\bar{K}_t)=(\widetilde{Y}_t-(\widetilde{A}_T-\widetilde{A}_t),\widetilde{Z}_t,\widetilde{K}_t), \ t\in[0,T].
\end{align*}
Therefore, $(\widetilde{Y},\widetilde{Z},\widetilde{K},\widetilde{A})\in \mathfrak{S}^{\alpha}_G(0,T)\times BV[0,T]$ is a solution to doubly mean reflected $G$-BSDE \eqref{nonlinearyz11}. For the last assertion,  as verified in Lemma 3.2 of \cite{L} (see also Remark \ref{mywA2} in the appendix), $\bar{Y}\in M_{G}^{\beta}\left(0, T \right)$, which together with $\widetilde{A}\in BV[0,T]$ implies $\widetilde{Y}\in M_{G}^{\beta}\left(0, T \right)$.

It remains to prove uniqueness. Let $(Y^i,Z^i,K^i,A^i)$, $i=1,2$ be the solution to doubly mean reflected $G$-BSDE with terminal value $\xi$, loss functions $L,R$ and coefficient $f$. Noting that $f$ is independent of $y$, $(Y^i-(A_T^i-A^i),Z^i,K^i)$, $i=1,2$, can be seen as the solution to $G$-BSDE with terminal value $\xi$ and coefficient $f$. Therefore, we have $Z^1\equiv Z^2=:Z$ and consequently, $(Y^i,Z^i,K^i,A^i)$, $i=1,2$ coincide with the solution to doubly mean reflected $G$-BSDE terminal value $\xi$, loss functions $L,R$ and constant coefficient $\{C_s\}=\{f(s, Z_s)\}$. By Proposition \ref{prop7}, we obtain the desired result.
\end{proof}

\subsection{The case of general Lipschitz coefficients}

In this subsection, we consider the doubly mean reflected $G$-BSDE with coefficients $f,g$ taking the following form
\begin{equation}\label{nonlinearyz1}
\begin{cases}
Y_t=\xi+\int_t^T f(s,Y_s,Z_s)ds+\int_t^T g(s,Y_s,Z_s)d\langle B\rangle_s-\int_t^T Z_s dB_s-(K_T-K_t)+(A_T-A_t), \\
\hE[L(t,Y_t)]\leq 0\leq \hE[R(t,Y_t)], \\
A_t=A^R_t-A^L_t, \ A^R,A^L\in I[0,T] \textrm{ and } \int_0^T \hE[R(t,Y_t)]dA_t^R=\int_0^T \hE[L(t,Y_t)]dA^L_t=0.
\end{cases}
\end{equation}
We assume that the coefficients $f,g$ satisfy the following assumptions.
\begin{assumption}\label{assfg}
\begin{itemize}
\item[(i)] for each $y,z$, $f(\cdot,y,z),g(\cdot,y,z)\in M_G^\beta(0,T)$ for some $\beta>1$;
\item[(ii)] there exists some constant $\kappa>0$, such that for any $t\in[0,T]$ and $y,y',z,z'\in\mathbb{R}$,
\begin{align*}
|f(t,y,z)-f(t,y',z')|+|g(t,y,z)-g(t,y',z')|\leq \kappa(|y-y'|+|z-z'|).
\end{align*}
\end{itemize}
\end{assumption}
\begin{theorem}\label{thm3.10}
Suppose that Assumption \ref{ass2} and \ref{assfg} hold. Let $\xi\in L_G^\beta(\Omega_T)$ be such that
$\hE[L(T,\xi)]\leq 0\leq \hE[R(t,\xi)]$. Then, for each $1< \alpha<\beta$, the $G$-BSDE with double mean reflections \eqref{nonlinearyz1} has a unique solution $(Y,Z,K,A)\in \mathfrak{S}^\alpha_G\times BV[0,T]$.
\end{theorem}
\begin{proof}
Without loss of generality, we assume $g \equiv 0$ in the proof for simplicity. We first consider the case that $T\leq \delta$, where $\delta$ is a constant small enough to be determined later. Given $U^i \in M_G^\beta(0,T)$, Theorem \ref{thm3.8} ensures that the following $G$-BSDE with double mean reflection admits a unique solution $(Y^i,Z^i,K^i,A^i)\in \mathfrak{S}^\alpha_G\times BV[0,T]$, and moreover $Y^i\in M_G^\beta(0,T)$,
\begin{equation}\label{A4}
\begin{cases}
Y^i_t=\xi+\int_t^T f(s,U^i_s,Z^i_s)ds-\int_t^T Z^i_s dB_s-(K^i_T-K^i_t)+(A^i_T-A^i_t), \\
\hE[L(t,Y^i_t)]\leq 0\leq \hE[R(t,Y^i_t)], \\
A^i_t=A^{R,i}_t-A^{L,i}_t, \ A^{R,i},A^{L,i}\in I[0,T] \textrm{ and } \int_0^T \hE[R(t,Y^i_t)]dA_t^{R,i}=\int_0^T \hE[L(t,Y^i_t)]dA^{L,i}_t=0.
\end{cases}
\end{equation}
Then, we define a map $\Gamma:M_G^\beta(0,T)\rightarrow M_G^\beta(0,T)$ as follows:
\begin{align*}
    \Gamma(U^i)=Y^i.
\end{align*}
 From the proof of Theorem \ref{thm3.8}, we could see that the following relationship holds:
\begin{align*}
(Y_{t}^i,Z_{t}^i,K_{t}^i)=(\bar{Y}_{t}^i+\left(A_{T}^i-A_{t}^i\right),\bar{Z}_{t}^i,\bar{K}_{t}^i),\ \forall t\in[0,T],
\end{align*}
where $(\bar{Y}^i,\bar{Z}^i,\bar{K}^i)$ is the solution to the  $G$-BSDE below:
\begin{align*}
\bar{Y}_{t}^i=\xi+\int_{t}^{T} f\left(s, U_{s}^{i},\bar{Z}_{s}^i\right) d s-\int_{t}^{T} \bar{Z}_{s}^i d B_{s}-\left(\bar{K}_{T}^i-\bar{K}_{t}^i\right).
\end{align*}
Thus,
\begin{align}\label{A5}
|Y_{t}^1-Y_{t}^2|\leq |\bar{Y}_{t}^1-\bar{Y}_{t}^2|+2\sup _{0 \leq t \leq T}|A_{t}^1-A_{t}^2|.
\end{align}
According Theorem \ref{mywA3}, there exists a constant $\widetilde{C}(\alpha,G,\kappa)>0$ such that
\begin{align}\label{A6}
|\bar{Y}_{t}^1-\bar{Y}_{t}^2|^{\alpha}\leq \widetilde{C}(\alpha,G,\kappa)e^{\widetilde{C}(\alpha,G,\kappa)T}\hE_t\bigg[\Big(\int_t^T|U_{s}^1-U_{s}^2|ds\Big)^{\alpha}\bigg].
\end{align}
Recalling Theorem \ref{BSP'} and the proof of Proposition \ref{prop7}, we have
\begin{align}\label{A7}
\sup_{t\in[0,T]}|A^1_t-A^2_t|\leq 2\frac{C}{c}|a^1-a^2|+4\frac{C}{c}\sup_{t\in[0,T]}|s^1_t-s^2_t|+2\frac{C}{c}\sup_{(t,x)\in[0,T]\times \mathbb{R}}(|\hat{L}(t,x)|\vee|\hat{R}(t,x)|),
\end{align}
where for $i=1,2$, $a^i=\hE[\xi]$ and
\begin{align*}
&s^i_t=\hE\bigg[\xi+\int_0^T f\left(s, U_{s}^{i},\bar{Z}_{s}^i\right) d s\bigg]-\hE\bigg[\xi+\int_t^T f\left(s, U_{s}^{i},\bar{Z}_{s}^i\right) d s\bigg]=\hE[\bar{Y}^i_0]-\hE[\bar{Y}^i_t],\\
&\hat{L}(t,x):=\hE[L(t,\bar{Y}^1_t-\hE[\bar{Y}^1_t]+x)]-\hE[L(t,\bar{Y}^2_t-\hE[\bar{Y}^2_t]+x)],\\
&\hat{R}(t,x):=\hE[R(t,\bar{Y}^1_t-\hE[\bar{Y}^1_t]+x)]-\hE[R(t,\bar{Y}^2_t-\hE[\bar{Y}^2_t]+x)].
\end{align*}
Combining \eqref{A5}-\eqref{A7} and by simple calculation, we could get
\begin{align*}
\hE\Big[|{Y}_{t}^1-{Y}_{t}^2|^{\beta}\Big]\leq \widetilde{C}(C,c,\beta,G,\kappa)e^{\widetilde{C}(\beta,G,\kappa)T}\hE\bigg[\Big(\int_0^T|U_{s}^1-U_{s}^2|ds\Big)^{\beta}\bigg].
\end{align*}
It follows from the H{\"o}lder inequality that
\begin{align*}
\hE\Big[|{Y}_{t}^1-{Y}_{t}^2|^{\beta}\Big]\leq \widetilde{C}(C,c,\beta,G,\kappa)e^{\widetilde{C}(\beta,G,\kappa)T}T^{\beta-1}\hE\bigg[\int_0^T|U_{s}^1-U_{s}^2|^{\beta}ds\bigg].
\end{align*}
Consequently,
$$
\left\|{Y}^{1}-{Y}^{2}\right\|_{M_G^\beta} \leq\left(\int_{0}^T \hat{\mathbb{E}}\left[\left|{Y}_t^{1}-{Y}_t^{2}\right|^\beta\right] d t\right)^{\frac{1}{\beta}} \leq \widetilde{C}(C,c,\beta,G,\kappa)e^{\widetilde{C}(\beta,G,\kappa)T}T\left\|U^1-U^2\right\|_{M_G^\beta}.
$$
Choosing $\delta$ small enough such that $\widetilde{C}(C,c,\beta,G,\kappa)e^{\widetilde{C}(\beta,G,\kappa)\delta}\delta<1$, we could conclude that the map $\Gamma(U):=Y$ from $M_G^\beta(0, T)$ to $M_G^\beta(0, T)$ is a contraction. Moreover, taking $U$ in \eqref{A4} as the fixed point of the map $\Gamma$, we have $Y\in S_G^\alpha(0, T)$. Therefore, the $G$-BSDE with double mean reflections \eqref{nonlinearyz1} has a unique solution on the small interval $[0,T]$.

For the general $T$, we choose $n\geq 1$ such that $n\delta \geq T$. Set $T_k:=\frac{k T}{n}$, for $k=0,1, \cdots, n$. By backward induction, for $k=n, n-1, \cdots, 1$, there exists a unique solution $\left(Y^k, Z^k, K^k,A^k\right)$ to the following $G$-BSDE with double mean reflection on the interval $\left[T_{k-1}, T_k\right]$
\begin{displaymath}
\begin{cases}
Y_t^k=Y_{T_k}^{k+1}+\int_t^{T_k} f\left(s, Y_s^k, Z_s^k\right) d s-\int_t^{T_k} Z_s^k d B_s-(K_{T_k}^k-K_t^k)+(A_{T_k}^k-A^k_t), \\
\hE[L(t, Y_t^k)] \leq 0 \leq \hE[R(t, Y_t^k)], t \in\left[T_{k-1}, T_k\right], \\
A_{T_{k-1}}^k=0, A_t^k=A_t^{k, R}-A_t^{k, L}, \int_{T_{k-1}}^{T_k} \hE[R(t, Y_t^k)] d A_t^{k, R}=\int_{T_{k-1}}^{T_k} \hE[L(t, Y_t^k)] d A_t^{k, L}=0,
\end{cases}
\end{displaymath}
where $Y_T^{n+1}=Y_T^{n}=\xi$. We denote
\begin{align*}
&Y_t=\sum_{k=1}^n Y_t^k I_{\left[T_{k-1}, T_k\right)}(t)+Y_T^{n} I_{\{T\}}(t), \  Z_t=\sum_{k=1}^n Z_t^k I_{\left[T_{k-1}, T_k\right)}(t)+Z_T^{n} I_{\{T\}}(t),\\
&K_t=K_t^k+\sum_{j=1}^{k-1} K_{T_j}^j, \ A_t=A_t^k+\sum_{j=1}^{k-1} A_{T_j}^j, \ t\in\left[T_{k-1}, T_k\right], \ k=1,\cdots, n,
\end{align*}
with the notation $\sum_{j=1}^{0} K_{T_j}^j=\sum_{j=1}^{0} A_{T_j}^j=0$ ($A^R, A^L$ are defined similarly). It is easy to check that $(Y,Z,K,A)\in \mathfrak{S}^\alpha_G\times BV[0,T]$ is  a  solution to $G$-BSDE \eqref{nonlinearyz1} with double mean reflections. The uniqueness follows immediately from the uniqueness on each small interval, which ends the proof.
\end{proof}

\subsection{The case of non-Lipschitz coefficients}
In this subsection, we shall investigate the well-posedness of the mean reflected $G$-BSDE \eqref{nonlinearyz1} under weaker conditions on the continuity property of the coefficients w.r.t. $y$. More precisely, we assume that $f, g$ satisfy (i) in Assumption \ref{assfg} and the following so-called $\beta$-order Mao's condition. 
\begin{assumption}\label{assfg2}
For some $\beta>1$, and any $y, y^{\prime} \in \mathbb{R},z, z^{\prime} \in \mathbb{R}^{d}$, $t\in [0, T]$
\begin{align*}
\begin{split}
\left|f(t, y, z)-f\left(t, y^{\prime},z^{\prime}\right)\right|^{\beta}+\left|g(t, y, z)-g\left(t, y^{\prime}, z^{\prime}\right)\right|^{\beta}\leq \rho(\left|y-y^{\prime}\right|^{\beta})+\kappa\left|z-z^{\prime}\right|^{\beta},
\end{split}
\end{align*}
where $\kappa$ is a positive constant, and  $\rho:[0,+\infty) \rightarrow[0,+\infty)$ is a continuous non-decreasing concave function with $\rho(0)=0$, $\rho(r)>0$ for $r>0$, such that $\int_{0+}\rho^{-1}(r)d r=+\infty$.
\end{assumption}

\begin{remark}
In the classical framework, Watanabe and Yamada \cite{YW, WY}  proved the pathwise uniqueness of solutions to SDEs under some similar non-Lipschitz condition in 1970s. Later, Mao \cite{Mao} studied the solvability of the BSDE whose generator $f$ satisfies $$\left|f(t, y, z)-f\left(t, y^{\prime}, z^{\prime}\right)\right|^{2} \leq \rho(\left|y-y^{\prime}\right|^{2})+\kappa\left|z-z^{\prime}\right|^{2}.$$ Notably, our condition presented in Assumption \ref{assfg2} is more general compared with the one in  \cite{Mao}.  Typical examples of $f,g$ satisfying such $\beta$-order Mao's condition can be found in Example 3.6 of \cite{H}.
\end{remark}
\begin{theorem}\label{myw310} Suppose that Assumption \ref{ass2}, Assumption \ref{assfg2} and (i) in Assumption \ref{assfg} hold. Let $\xi\in L_G^\beta(\Omega_T)$ be such that
$\hE[L(T,\xi)]\leq 0\leq \hE[R(t,\xi)]$. Then the $G$-BSDE with double mean reflections \eqref{nonlinearyz1} has a unique solution $(Y,Z,K,A)\in \mathfrak{S}^\alpha_G\times BV[0,T]$ for each $1< \alpha<\beta$.
\end{theorem}

Since there is no appropriate estimate for the $Z$ component (see Remark 3.15 in \cite{LiuW} for details), we cannot construct a contraction mapping involving $Z$. As a consequence, in order to prove Theorem \ref{myw310}, we use the Picard iteration argument only for the $Y$ component. The approximate sequence is constructed recursively as follows:  Let $Y_{t}^0 \equiv 0$, for any $ t \in [0, T]$, and we define
\begin{equation*}
\begin{cases}
Y^n_t=\xi+\int_t^T f(s,Y^{n-1}_s,Z^n_s)ds+\int_t^T g(s,Y^{n-1}_s,Z^n_s)d\langle B\rangle_s-\int_t^T Z^n_s dB_s-(K^n_T-K^n_t)+(A^n_T-A^n_t), \\
\hE[L(t,Y^n_t)]\leq 0\leq \hE[R(t,Y^n_t)], \\
A^n_t=A^{R,n}_t-A^{L,n}_t, \ A^{R,n},A^{L,n}\in I[0,T] \textrm{ and } \int_0^T \hE[R(t,Y^n_t)]dA_t^{R,n}=\int_0^T \hE[L(t,Y^n_t)]dA^{L,n}_t=0.
\end{cases}
\end{equation*}
The sequence $\{{Y}^n\}$ is well-defined in  $M_{G}^{\beta}(0, T)$ according to Lemma 2.15 in \cite{H} and Theorem \ref{thm3.8}. Furthermore, we have the following relationship from the proof of Theorem \ref{thm3.8}:
\begin{align}\label{309}
(Y_{t}^n,Z_{t}^n,K_{t}^n)=(\bar{Y}_{t}^n+\left(A_{T}^n-A_{t}^n\right),\bar{Z}_{t}^n,\bar{K}_{t}^n),\ \forall t\in[0,T],
\end{align}
where $(\bar{Y}^n,\bar{Z}^n,\bar{K}^n)$ is the solution to the following $G$-BSDE
\begin{align}\label{310}
\bar{Y}_{t}^n=\xi+\int_{t}^{T} f\left(s, Y_{s}^{n-1}, \bar{Z}_{s}^n\right) d s+\int_t^T g(s,Y^{n-1}_s,\bar{Z}_{s}^n)d\langle B\rangle_s-\int_{t}^{T} \bar{Z}_{s}^n d B_{s}-\left(\bar{K}_{T}^n-\bar{K}_{t}^n\right).
\end{align}

Firstly, we need to give some a prior estimates, which are crucial for the subsequent discussions. Without loss of generality, we assume $g \equiv 0$ in what follows  for simplicity.
\begin{lemma}\label{myw311} Under the same assumptions as in Theorem \ref{myw310}, we could get
\begin{description}
\item[(i)] There exists a positive constant $\widetilde{C}(\beta, T, G, \kappa,c,C,\rho)$ such that for all $n\geq 1$
\begin{align*}
\begin{split}
\sup_{0 \leqslant t \leqslant T}\hat{\mathbb{E}}\left[\left|Y_t^{n}\right|^{\beta}\right]\leq \widetilde{C}(\beta, T, G, \kappa,c,C,\rho).
\end{split}
\end{align*}
\item[(ii)]

$\sup\limits_{0 \leqslant t \leqslant T}\hat{\mathbb{E}}\left[\left|Y_t^{n+m}-Y_t^{n}\right|^{\beta}\right]\rightarrow 0, \  as\ n, m \rightarrow \infty.$

\end{description}
\end{lemma}
\begin{proof} The proof relies heavily on the representation \eqref{309}-\eqref{310} and the basic estimates of $G$-BSDEs in \cite{HJPS1} (see the appendix for details). The main difficulty lies in deriving some suitable estimates related to the constraining process $A$.

{\bf Proof of Assertion (i):}
Since $\rho$ is concave and $\rho(0)=0,$ one can find a pair of positive constants $a$ and
$b$ such that
\begin{align*}
\rho(u) \leqslant a+b u, \quad \text { for all } u \geqslant 0.
\end{align*}
So, there exists a constant $\widetilde{C}(\rho,\beta)$ such that
\begin{align}\label{A311}
\left|f(s, {Y}_s^{n-1},0)\right|^{\beta}\leq \widetilde{C}(\rho,\beta)\Big(1+\left|f(s, 0, 0)\right|^{\beta}+|{Y}_s^{n-1}|^{\beta}\Big).
\end{align}
Making full use of \eqref{AA1} in the appendix and by a similar argument as the proof of Lemma 3.3 in \cite{H}, we deduce that
\begin{align}\label{A312}
\begin{split}
\left|\bar{Y}_t^{n}\right|^{\beta}\leq \widetilde{C}(\beta, T, G, \kappa,\rho)\hat{\mathbb{E}}_{t}\bigg[1+|\xi|^{\beta}+\int_{t}^{T}\Big(\left|f(s, 0, 0)\right|^{\beta}+|{Y}_s^{n-1}|^{\beta}\Big)d s\bigg],
\end{split}
\end{align}
Moreover, recalling Theorem \ref{AA209} and the proof of Proposition \ref{prop7}, we have
\begin{align}\label{AA312}
\begin{split}
\left|A_T^{n}-A_t^{n}\right|^{\beta}&\leq \widetilde{C}(\beta, C,c)\Big(\big|\hE[\xi]+\hE[\bar{Y}_0^{n}]-\hE[\xi]-\hE[\bar{Y}_0^{n}]+\hE[\bar{Y}_t^{n}] \big|^{\beta} \\
&\ \ \ \ \ \ \ \ \ \ \ \ \ +\sup_{t\in[0,T]}\big(\hE[|L(t,0)|^\beta]+\hE[|R(t,0)|^\beta]\big)\Big)\\
&\leq \widetilde{C}(\beta, T, C,c)\Big(\hE[|\bar{Y}_t^{n}|^{\beta}] +1\Big),
\end{split}
\end{align}
which together with \eqref{309} gives
\begin{align*}
\begin{split}
\hat{\mathbb{E}}\Big[\big|{Y}_t^{n}\big|^{\beta}\Big]\leq \widetilde{C}(\beta, T, G, \kappa,c,C,\rho)\Bigg\{1+\int_{t}^{T}\hat{\mathbb{E}}\left[|{Y}_s^{n-1}|^{\beta}\right]d s\Bigg\}.
\end{split}
\end{align*}
Set $p(t)=\widetilde{C}(\beta, T, G, \kappa,c,C,\rho)e^{\widetilde{C}(\beta, T, G, \kappa,c,C,\rho)(T-t)}$ and $p(\cdot)$ solves the following ODE,
\begin{align*}
p(t)=\widetilde{C}(\beta, T, G, \kappa,c,C,\rho)\bigg(1+\int_{t}^{T}p(s)d s\bigg).
\end{align*}
Hence, it is easy to verify by recurrence that for any $n\geq1$,
\begin{align*}
\hat{\mathbb{E}}\left[\left|Y_t^{n}\right|^{\beta}\right]\leq p(t),\ t \in [0, T],
\end{align*}
which completes the proof.

{\bf Proof of Assertion (ii):} For any $n, m\geq 1$, according to \eqref{AA3} in the appendix and using the same technique as \eqref{A312}, we could find a constant $\widetilde{C}(\beta, T, G, \kappa)>0$ such that
\begin{align}\label{AA3312}
\begin{split}
\hat{\mathbb{E}}\left[\left|\bar{Y}_t^{n+m}-\bar{Y}_t^{n}\right|^{\beta}\right]&\leq \widetilde{C}(\beta, T, G, \kappa)\int_{t}^{T}\hat{\mathbb{E}}\Big[\rho\Big(\left|{Y}_s^{n+m-1}-{Y}_s^{n-1}\right|^{\beta}\Big)\Big]d s\\
&\leq \widetilde{C}(\beta, T, G, \kappa)\int_{t}^{T}\rho\Big(\hat{\mathbb{E}}\Big[\left|{Y}_s^{n+m-1}-{Y}_s^{n-1}\right|^{\beta}\Big]\Big)d s,
\end{split}
\end{align}
where we have used Jensen's inequality in $G$-framework (see Lemma \ref{myw210}) in the second inequality. Getting insight from the proof of Proposition 3.14 in \cite{Li2}, the proof of Theorem 3.9 in \cite{Li} and the proof of Theorem \ref{thm3.10}, it is worth noting that 
\begin{equation}\label{AA313'}
\begin{split}
&\Big|(A_{T}^{n+m}-A_{t}^{n+m})-(A_{T}^n-A_{t}^n)\Big|=|\widetilde{k}^{n+m}_{T-t}-\widetilde{k}^{n}_{T-t}|\\
\leq &\frac{C}{c}\sup_{v\in[0,T-t]}|\widetilde{s}^{n+m}_v-\widetilde{s}^{n}_v|+\frac{1}{c}\sup_{(v,x)\in[0,T-t]\times \mathbb{R}}\Big(|\widehat{\widetilde{l}}^{m,n}(v,x)|\vee|\widehat{\widetilde{r}}^{m,n}(v,x)|\Big)\\
=&\frac{C}{c}\sup_{v\in[0,T-t]}\Big|{s}^{n+m}_T-{s}^{n+m}_{T-v}-\big({s}^{n}_T-{s}^{n}_{T-v}\big)\Big|+\frac{1}{c}\sup_{(v,x)\in[t,T]\times \mathbb{R}}\Big(|\widehat{l}^{m,n}(v,x)|\vee|\widehat{r}^{m,n}(v,x)|\Big)\\
\leq& \frac{C}{c}\sup_{v\in[0,T-t]}\Big(\Big|\hE[\bar{Y}_T^{n+m}]-\hE[\bar{Y}_T^{n}]\Big|+ \Big|\hE[\bar{Y}_{T-v}^{n+m}]-\hE[\bar{Y}_{T-v}^{n}]\Big|\Big)+\frac{C}{c}\sup_{v\in[t,T]}\hat{\mathbb{E}}\left[\left|\bar{Y}_v^{n+m}-\bar{Y}_v^{n}\right|\right]\\
\leq&\widetilde{C}(C,c)\sup_{v\in[t,T]}\hat{\mathbb{E}}\left[\left|\bar{Y}_v^{n+m}-\bar{Y}_v^{n}\right|\right],
\end{split}
\end{equation}
where, for any $n\geq 1$, $a^n=\hE[\xi]$ and for any $t\in[0,T]$ 
\begin{align*}
&s^n_t=\hE\bigg[\xi+\int_0^T f\left(s, Y_{s}^{n-1},\bar{Z}_{s}^n\right) d s\bigg]-\hE\bigg[\xi+\int_t^T f\left(s, Y_{s}^{n-1},\bar{Z}_{s}^n\right) d s\bigg]=\hE[\bar{Y}^n_0]-\hE[\bar{Y}^n_t],\\
&l^n(t,x)=\hE[L(t,\bar{Y}^n_t-\hE[\bar{Y}^n_t]+x)],\ r^n(t,x)=\hE[R(t,\bar{Y}^n_t-\hE[\bar{Y}^n_t]+x)],\\
& \widetilde{s}^{n}_t=a^n+s_T^n-s_{T-t}^n,\ \widetilde{l}^{n}(t,x)=l^n(T-t,x),\ \widetilde{r}^{n}(t,x)=r^n(T-t,x), \\
&\widehat{h}^{m,n}(t,x)=h^{n+m}(t,x)-h^n(t,x), \textrm{ for } h=l,r,\widetilde{l},\widetilde{r},
\end{align*}
and $(\widetilde{x}^n,\widetilde{k}^n)$ denotes the unique solution to the Skorokhod problem $\mathbb{SP}_{\widetilde{l}^n}^{\widetilde{r}^n}(\widetilde{s}^n)$. Thus, it follows from the representation \eqref{309} that
\begin{align}\label{A313}
\begin{split}
\hat{\mathbb{E}}\left[\left|{Y}_t^{n+m}-{Y}_t^{n}\right|^{\beta}\right]&\leq \widetilde{C}(\beta, T, G, \kappa,c,C)\int_{t}^{T}\rho\Big(\hat{\mathbb{E}}\Big[\left|{Y}_s^{n+m-1}-{Y}_s^{n-1}\right|^{\beta}\Big]\Big)d s.
\end{split}
\end{align}
Furthermore, the above inequality can be rewritten as
\begin{align*}
u_{n, m}(t)\leq \widetilde{C}(\beta, T, G, \kappa,c,C)\int_{t}^{T}\rho(u_{n-1, m}(s))d s,
\end{align*}
where $u_{n, m}(t)=\sup_{t \leqslant r \leqslant T}\hat{\mathbb{E}}\left[\left|Y_r^{n+m}-Y_r^{n}\right|^{
\beta}\right]$ is uniformly bounded by Assertion (i). Set $v_n(t)=\sup_{m}u_{n, m}(t)$ and $\alpha(t)=\limsup _{n\rightarrow+\infty} v_{n}(t)$. Applying Lebesgue's dominated convergence theorem, we could get
\begin{align*}
0\leq \alpha(t)\leq\widetilde{C}(\beta, T, G, \kappa,c,C)\int_{t}^{T}\rho\Big(\alpha(s)\Big)d s, \ 0 \leq t \leq T.
\end{align*}
The final result is immediate from the backward Bihari's inequality.
\end{proof}

Now we are ready to present the proof of Theorem \ref{myw310}, which is the main result of this subsection.
\begin{proof}[Proof of  Theorem \ref{myw310}] We first prove the existence, which will be divided into the following three steps.

{\bf Step 1. The uniform estimates.} We notice that there exists $\theta>0$ such that $\alpha(1+\theta)=\beta$. Consequently, in view of Theorem \ref{myw204} and \eqref{AA1} in the appendix, we could find a positive constant $\widetilde{C}( \beta, T, G, \kappa)$ such that
\begin{align}\label{A314}
\begin{split}
\hat{\mathbb{E}}\left[\sup_{0 \leqslant t \leqslant T}\left|\bar{Y}_t^{n}\right|^{\alpha}\right]&\leq \widetilde{C}( \beta, T, G, \kappa)\bigg(\hat{\mathbb{E}}\bigg[\Big||\xi|^{\alpha}+\int_{0}^{T}\left|f(s, Y_s^{n-1}, 0)\right|^{\alpha}d s\Big|^{1+\theta}\bigg]\bigg)^{\frac{1}{(1+\theta)}}.
\end{split}
\end{align}
Observe that from Assertion (i) in Lemma \ref{myw311} and \eqref{A311}, we could get
\begin{align*}
\begin{split}
&\hat{\mathbb{E}}\bigg[\Big||\xi|^{\alpha}+\int_{0}^{T}\left|f(s, {Y}_s^{n-1}, 0)\right|^{\alpha}d s\Big|^{1+\theta}\bigg]\\
\leq &\widetilde{C}(\beta, T)\Bigg\{\hat{\mathbb{E}}\left[|\xi|^{\beta}\right]+\hat{\mathbb{E}}\bigg[\int_{0}^{T}\left|f(s, {Y}_s^{n-1}, 0)\right|^{\beta}d s\bigg]\Bigg\}\\
\leq &\widetilde{C}(\beta, T,\rho)\Bigg\{ \hat{\mathbb{E}}\left[|\xi|^{\beta}\right]+\sup_{0 \leqslant t \leqslant T}\hat{\mathbb{E}}\left[\left|{Y}_t^{n-1}\right|^{\beta}\right]+1+\hat{\mathbb{E}}\bigg[\int_{0}^{T}\left|f(s, 0, 0)\right|^{\beta}d s \bigg] \Bigg\}\\
\leq& \widetilde{C}(\beta, T, G, \kappa,c,C,\rho).
\end{split}
\end{align*}
Moreover, recalling \eqref{AA312}, we have
\begin{align}\label{A316}
\begin{split}
\sup _{0 \leq t \leq T}|A_{T}^n-A_{t}^n|^{\alpha}\leq \widetilde{C}(\alpha,T,c,C)\Big(1+\hat{\mathbb{E}}\Big[\sup _{0 \leq t \leq T}\left|\bar{Y}_t^{n}\right|^{\alpha}\Big]\Big).
\end{split}
\end{align}
Therefore, we can conclude from the representation \eqref{309} and Assertion (i) in Lemma \ref{myw311} that for any $n\geq1$,
\begin{align}\label{A317}
\hat{\mathbb{E}}\left[\sup_{0 \leqslant t \leqslant T}\left|Y_t^{n}\right|^{\alpha}\right]&\leq\widetilde{C}(\beta, T, G, \kappa,c,C,\rho).
\end{align}
We recall that $({Z}^n,{K}^n)$ can be seen as a part of the solution to $G$-BSDE \eqref{310}. Therefore, using \eqref{AA2} in the appendix  together with the H{\"o}lder inequality, we derive that
\begin{align}\label{A318}
\begin{split}
&\|{{Z}^n}\|_{H_{G}^{\alpha}}^{\alpha}+\|{K}^n_{T}\|_{L_{G}^{\alpha}}^{\alpha}=\|{\bar{Z}^n}\|_{H_{G}^{\alpha}}^{\alpha}+\|\bar{K}^n_{T}\|_{L_{G}^{\alpha}}^{\alpha}\\
\leq &\widetilde{C}(\alpha, T, G, \kappa) \Bigg\{\left\|{\bar{Y}^n}\right\|_{S_{G}^{\alpha}}^{\alpha}+\hat{\mathbb{E}}\bigg[\Big|\int_{0}^{T} \Big|{f}\left(s, {Y}_s^{n-1},  0\right)\Big|d s\Big|^{\alpha}\bigg]\Bigg\}\\
\leq &\widetilde{C}(\alpha, T, G, \kappa) \Bigg\{\left\|{\bar{Y}^n}\right\|_{S_{G}^{\alpha}}^{\alpha}+\hat{\mathbb{E}}\bigg[\Big|\int_{0}^{T} \Big|{f}\left(s, {Y}_s^{n-1}, 0\right)\Big|^{\beta}d s\Big|^{\frac{\alpha}{\beta}}\bigg]\Bigg\}\\
\leq &\widetilde{C}(\alpha, T, G, \kappa,\rho) \Bigg\{\|{\bar{Y}^n}\|_{S_{G}^{\alpha}}^{\alpha}+\|{{Y}^{n-1}}\|_{S_{G}^{\alpha}}^{\alpha}+1+\hat{\mathbb{E}}\bigg[\Big|\int_{0}^{T} \big|{f}\left(s,  0,0\right)\big|^{\beta}d s\Big|^{\frac{\alpha}{\beta}}\bigg]\Bigg\}\\
\leq &\widetilde{C}(\beta, T, G, \kappa,c,C,\rho).
\end{split}
\end{align}

{\bf Step 2. The convergence.} We start from showing that $\{{Y}^n\}_{n\in\mathbb{N}}$ is a Cauchy sequence in $S_G^{\alpha}(0, T)$. By  \eqref{AA3} in the appendix, the following inequality can be deduced in a similar way to \eqref{A314}. Indeed,
\begin{align*}
\begin{split}
\hat{\mathbb{E}}\bigg[\sup_{0 \leqslant t \leqslant T}\left|\bar{Y}_t^{n+m}-\bar{Y}_t^{n}\right|^{\alpha}\bigg]&\leq\widetilde{C}(\beta, T,G,\kappa)\bigg(\hat{\mathbb{E}}\bigg[\Big|\int_{0}^{T}\left|f_s^{n+m-1}-f_s^{n-1}\right|^{\alpha}d s\Big|^{1+\theta}\bigg]\bigg)^{\frac{1}{(1+\theta)}},
\end{split}
\end{align*}
where $f_s^{n+m-1}-f_s^{n-1}=f(s, Y_s^{n+m-1},\bar{Z}_s^{n+m})-f(s, Y_s^{n-1}, \bar{Z}_s^{n+m})$. The following inequalities comes from the H{\"o}lder inequality and Assumption \ref{assfg2},
\begin{align*}
\begin{split}
\hat{\mathbb{E}}\bigg[\Big|\int_{0}^{T}\left|f_s^{n+m-1}-f_s^{n-1}\right|^{\alpha}d s\Big|^{1+\theta}\bigg]
&\leq \widetilde{C}(\beta, T)\hat{\mathbb{E}}\bigg[\int_{0}^{T}\left|f_s^{n+m-1}-f_s^{n-1}\right|^{{\alpha}(1+\theta)}d s\bigg]\\
&\leq \widetilde{C}(\beta, T)\hat{\mathbb{E}}\bigg[\int_{0}^{T}\rho\Big(| {Y}_s^{n+m-1}-{Y}_s^{n-1}|^{\beta}\Big)d s\bigg]\\
&\leq \widetilde{C}(\beta, T)\rho\Big(\sup_{0 \leqslant s \leqslant T}\hat{\mathbb{E}}\left[| {Y}_s^{n+m-1}-{Y}_s^{n-1}|^{\beta}\right]\Big),
\end{split}
\end{align*}
where in the last inequality we have used Lemma \ref{myw210}. Moreover, recalling \eqref{AA313'}, we could derive that 
\begin{align*}
\begin{split}
\sup_{0 \leqslant t \leqslant T}\Big|(A_{T}^{n+m}-A_{t}^{n+m})-(A_{T}^n-A_{t}^n)\Big|^{\alpha}\leq\widetilde{C}(\alpha,C,c)\sup_{t\in[0,T]}\hat{\mathbb{E}}\left[\left|\bar{Y}_t^{n+m}-\bar{Y}_t^{n}\right|^{\alpha}\right].
\end{split}
\end{align*}
All the above conclusions along with the representation \eqref{309} and Lemma \ref{myw311} imply that $Y^n$ is a Cauchy sequence in $S_{G}^{\alpha}(0, T) $.  As for the ${Z}$-component, with the help of  \eqref{AA4} in the appendix and the same techniques as \eqref{A318}, we could get 
\begin{align*}
\begin{split}
&\|{{Z}}^{n+m}-{{Z}}^{n}\|_{H_{G}^{\alpha}}^{\alpha}=\|{\bar{Z}}^{n+m}-{\bar{Z}}^{n}\|_{H_{G}^{\alpha}}^{\alpha}\\
\leq &\widetilde{C}(\beta, T, G, \kappa,\rho)\Bigg\{\|\bar{Y}^{n+m}-\bar{Y}^{n}\|_{S_{G}^{\alpha}}^{\alpha}+\|\bar{Y}^{n+m}-\bar{Y}^{n}\|_{S_{G}^{\alpha}}^{\frac{\alpha}{2}} \bigg[\left\|\bar{Y}^{n+m}\right\|_{S_{G}^{\alpha}}^{\frac{\alpha}{2}}\\
&+\left\|{Y}^{n+m-1}\right\|_{S_{G}^{\alpha}}^{\frac{\alpha}{2}}+\left\|\bar{Y}^{n}\right\|_{S_{G}^{\alpha}}^{\frac{\alpha}{2}}+\left\|{Y}^{n-1}\right\|_{S_{G}^{\alpha}}^{\frac{\alpha}{2}}+1+\hat{\mathbb{E}}\Big[\Big|\int_{0}^{T}|{f}\left(s, 0,0\right)|^{\beta}d s\Big|^{\frac{\alpha}{\beta}}\Big]^{\frac{1}{2}}\bigg]\Bigg\},
\end{split}
\end{align*}
which indicates $\{{Z}^n\}$ is a Cauchy sequence in $H_G^{\alpha}(0, T)$. Besides, due to \eqref{AA313'}, we have
\begin{equation}\label{A319}
\begin{split}
\sup_{0 \leqslant t \leqslant T}|A_t^{n+m}-A_t^n|^{\alpha}
&\leq \sup_{0 \leqslant t \leqslant T}\Big|(A_{T}^{n+m}-A_{t}^{n+m})-(A_{T}^n-A_{t}^n)\Big|^{\alpha}+\Big|A_{T}^{n+m}-A_{T}^n\Big|^{\alpha}\\
&\leq \widetilde{C}(\alpha,c,C)\sup _{0 \leq t \leq T}\hat{\mathbb{E}}\left[\left|\bar{Y}^{n+m}_t-\bar{Y}^n_t\right|^{\alpha}\right] .
\end{split}
\end{equation}
Consequently, there are three processes $(Y, Z,A) \in S_{G}^{\alpha}(0, T) \times H_{G}^{\alpha}(0, T)\times BV[0,T]$ such that
\[
\lim _{n \rightarrow \infty}\left(\left\|Y^{n}-Y\right\|_{S_{G}^{\alpha}}^{\alpha}+\left\|Z^{n}-Z\right\|_{H_{G}^{\alpha}}^{\alpha}+\sup_{0 \leqslant t \leqslant T}|A_t^n-A_t|^{\alpha}\right)=0.
\]
 Note that $\lim _{n \rightarrow \infty}\left\|Y^{n}-Y\right\|_{S_{G}^{\alpha}}^{\alpha}=0$ means $Y^n$ converges to $Y$ pointwisely, which together with (ii) in Lemma \ref{myw311} gives $\lim _{n \rightarrow \infty}\sup\limits_{0 \leqslant t \leqslant T}\hat{\mathbb{E}}\left[\left|Y_t^{n}-Y_t\right|^{\beta}\right]=0$.  Then, we define 
\[
K_{t}=Y_{t}-Y_{0}+\int_{0}^{t} f\left(s, Y_{s}, Z_{s}\right) d s-\int_{0}^{t} Z_{s} d B_{s}+A_t.
\]
We claim that
\begin{align}\label{A320}
\lim _{n \rightarrow \infty}\|{K}^{n}-{K}\|_{S_{G}^{\alpha}}^{\alpha}=0,
\end{align}
which will be proved in step 3. Note that ${K}^{n}=\bar{K}^{n}$ in \eqref{309} and $\bar{K}^{n}$ is a non-increasing $G$-martingale, which together with \eqref{A320} ensure ${K}$ is a non-increasing $G$-martingale. Then $({Y}, {Z}, {K}, {A}) \in \mathfrak{S}_G^\alpha \times BV[0,T]$ for any $1<\alpha<\beta$ is actually the solution to mean reflected $G$-BSDE \eqref{nonlinearyz1} with non-Lipschitz cofficients.

{\bf Step 3.} In fact,
\begin{align*}
\begin{split}
{K}_t^{n}-{K}_t={Y}_t^n-{Y}_t-({Y}_{0}^{n}-{Y}_{0})&+\int_{0}^{t}\Big(f(s,{Y}_s^{n-1},{Z}_s^{n})-f(s,{Y}_s,{Z}_s) \Big)d s\\
&-\int_{0}^{t}({Z}_s^{n}-{Z}_s)dB_s+{A}_t^n-{A}_t.
\end{split}
\end{align*}
Then, by Proposition \ref{BDG}, Lemma \ref{myw210} and simple calculation, we arrive at
\begin{align*}
\begin{split}
&\|{K}^{n}-{K}\|_{S_G^\alpha}^\alpha\leq \widetilde{C}\Bigg\{\|{Y}^{n}-{Y}\|_{S_G^\alpha}^\alpha+\hat{\mathbb{E}}\bigg[\Big(\int_{0}^{T}|f(s,{Y}_s^{n-1},{Z}_s^{n})-f(s,{Y}_s,{Z}_s^{n})|^{\beta} d s\Big)^{\frac{\alpha}{\beta}}\bigg]\\
&\ \ \ \ \ \  +\hat{\mathbb{E}}\bigg[\Big(\int_{0}^{T}|f(s,{Y}_s,{Z}_s^{n})-f(s,{Y}_s,{Z}_s)|^{2} d s\Big)^{\frac{\alpha}{2}}\bigg]+\|{Z}^{n}-{Z}\|_{H_G^\alpha}^\alpha+\sup_{0 \leqslant t \leqslant T}|{A}_t^{n}-{A}_t|^{\alpha}\Bigg\}\\
&\leq \widetilde{C}\Bigg\{\|{Y}^{n}-{Y}\|_{S_G^\alpha}^\alpha+\hat{\mathbb{E}}\bigg[\Big(\int_{0}^{T}\rho \big(|{Y}_s^{n-1}-{Y}_s|^{\beta} \big)d s\Big)^{\frac{\alpha}{\beta}}\bigg]+\|{Z}^{n}-{Z}\|_{H_G^\alpha}^\alpha+\sup_{0 \leqslant t \leqslant T}|{A}_t^{n}-{A}_t|^{\alpha}\Bigg\}\\
&\leq \widetilde{C}\Bigg\{\|{Y}^{n}-{Y}\|_{S_G^\alpha}^\alpha+\|{Z}^{n}-{Z}\|_{H_G^\alpha}^\alpha +\bigg|\rho\bigg(\sup_{0 \leqslant s \leqslant T}\hat{\mathbb{E}}\Big[|{Y}_s^{n-1}-{Y}_s|^{\beta}\Big]\bigg)\bigg|^{\frac{\alpha}{\beta}}+\sup_{0 \leqslant t \leqslant T}|{A}_t^{n}-{A}_t|^{\alpha}\Bigg\},
\end{split}
\end{align*}
where $\widetilde{C}=\widetilde{C}(\beta, T, G,\kappa,c,C)$. Thus the claim \eqref{A320} holds.

It remains to prove Uniqueness.  Assume that $({Y}^{\prime}, {Z}^{\prime}, {K}^{\prime}, {A}^{\prime})$ is another  solution to mean reflected $G$-BSDE \eqref{nonlinearyz1}. Similar with \eqref{A313}, we can derive that
\begin{align*}
\begin{split}
\hat{\mathbb{E}}\left[|{Y}_t-{Y}_t^{\prime}|^{\beta}\right]\leq \widetilde{C}(\beta, T, G, \kappa,c,C)\int_{t}^{T}\rho\bigg(\hat{\mathbb{E}}\Big[|{Y}_s-{Y}_s^{\prime}|^{\beta}\Big]\bigg)d s.
\end{split}
\end{align*}
It follows from the backward Bihari's inequality that
\begin{align*}
\hat{\mathbb{E}}\left[|{Y}_t-{Y}_t^{\prime}|^{\beta}\right]=0, \ t \in [0, T].
\end{align*}
From the continuity of ${Y}$ and ${Y}^{\prime}$, we could conclude that ${Y}_t={Y}_t^{\prime}$, $0\leq t\leq T$, q.s.. Now observe that $(Y, Z, K,A)$ and $(Y', Z', K',A')$ can both be seen as the  solutions to the following mean reflected $G$-BSDE on the time interval $[0, T]$
\begin{align*}
\begin{cases}
\bar{Y}_t=\xi+\int_t^T f(s,Y_s,\bar{Z}_s)ds-\int_t^T \bar{Z}_s dB_s-(\bar{K}_T-\bar{K}_t)+(\bar{A}_T-\bar{A}_t), \\
\hE[L(t,\bar{Y}_t)]\leq 0\leq \hE[R(t,\bar{Y}_t)], \\
\bar{A}_t=\bar{A}^R_t-\bar{A}^L_t, \ \bar{A}^R,\bar{A}^L\in I[0,T] \textrm{ and } \int_0^T \hE[R(t,\bar{Y}_t)]d\bar{A}_t^R=\int_0^T \hE[L(t,\bar{Y}_t)]d\bar{A}^L_t=0.
\end{cases}
\end{align*}
The uniqueness follows from Theorem \ref{thm3.8} and the fact $Y\in M_{G}^{\beta}(0, T)$.
\end{proof}

\section{Properties of solutions to $G$-BSDEs with double mean reflections}

In this section, we shall investigate some properties of the  solutions to $G$-BSDEs with double mean reflections which may bring some inspiration to the stochastic control problems and financial issues in our future research.
\subsection{Comparison theorem}
First, it is natural to conjecture that if
the loss function $R$ is larger, the force aiming to push the solution upwards is less and similarly, if the loss function $L$ is smaller, the force aiming to pull the solution downwards is less. To be more precisely, we have the following comparison theorem for the solutions to doubly mean reflected $G$-BSDEs with respect to the loss functions.

\begin{proposition}\label{prop11'}
Let the coefficients $f,g$ satisfy Assumption \ref{assfg1}. Suppose that $\xi\in L^\beta(\Omega_T)$ and $L^i,R^i$ satisfy Assumption \ref{ass2} with $\hE[L^i(T,\xi)]\leq 0\leq \hE[R^i(T,\xi)]$, $i=1,2$. Let $(Y^i,Z^i,K^i,A^i)$ be the solution to the doubly mean reflected $G$-BSDE with parameters $(\xi,f,g,L^i,R^i)$, $i=1,2$, respectively. Assume furthermore that for any $(t,x)\in[0,T]\times\mathbb{R}$, $L^1(t,x)\leq L^2(t,x)$, $R^1(t,x)\leq R^2(t,x)$, q.s. Then, for any $t\in[0,T]$, we have $Y^2_t\leq Y^1_t$.
\end{proposition}

\begin{proof}
Recalling the proof of Theorem \ref{thm3.8}, it suffices to  consider the case where $g=0$ and $f$ does not depends on $y$. First, note that $(Y^i-(A^i_T-A^i),Z^i,K^i)$ are the solutions to $G$-BSDE with terminal value $\xi$ and coefficient $f$, $i=1,2$. It follows from the uniqueness result for $G$-BSDEs that
\begin{align}\label{equation25}
Y^1_t-(A^1_T-A^1_t)=Y^2_t-(A^2_T-A^2_t), \ t\in[0,T].
\end{align}
We claim that $A^2_T-A^2_t\leq A^1_T-A^1_t$ for any $t\in[0,T]$, which, together with \eqref{equation25}, implies the desired result. We prove this claim by way of contradiction. Suppose  that there exists some $t_1<T$, such that
\begin{align*}
A^2_T-A^2_{t_1}> A^1_T-A^1_{t_1}.
\end{align*}
We define
\begin{align*}
t_2:=\inf\{t\geq t_1: A^2_T-A^2_t\leq A^1_T-A^1_t\}.
\end{align*}
Since $A^i$, $i=1,2$ are continuous, it is easy to check that
\begin{align}\label{equation26}
A^2_T-A^2_{t_2}= A^1_T-A^1_{t_2}, \ A^2_T-A^2_t> A^1_T-A^1_t, \ t\in [t_1,t_2).
\end{align}
Recalling Eq. \eqref{equation25}, we derive that $Y^2_t>Y^1_t$, $t\in[t_1,t_2)$, which implies that
\begin{align*}
&\hE[R^2(t,Y^2_t)]\geq \hE[R^1(t,Y^2_t)]>\hE[R^1(t,Y^1_t)]\\
\geq &0\geq \hE[L^2(t,Y^2_t)]\geq \hE[L^1(t,Y^2_t)]>\hE[L^1(t,Y^1_t)], \ t\in[t_1,t_2),
\end{align*}
where the strict inequality follows by a similar analysis as in the proof of Lemma \ref{proofofass1}. By the Skorokhod condition, we have $d A^{2,R}_t=d A^{1,L}_t=0$, $t\in[t_1,t_2)$. Simple calculation yields that
\begin{align*}
&A^1_T-(A^{1,R}_{t_1}-A^{1,L}_{t_2})=A^1_T-(A^{1,R}_{t_1}-A^{1,L}_{t_1})=A^1_T-A^1_{t_1}\\
<&A^2_T-A^2_{t_1}=A^2_T-(A^{2,R}_{t_1}-A^{2,L}_{t_1})=A^2_T-(A^{2,R}_{t_2}-A^{2,L}_{t_1})\leq A^2_T-(A^{2,R}_{t_2}-A^{2,L}_{t_2})\\
=&A^2_T-A^2_{t_2}=A^1_T-A^1_{t_2}=A^1_T-(A^{1,R}_{t_2}-A^{1,L}_{t_2}),
\end{align*}
where we have used \eqref{equation26} and the fact that $A^{2,L}$ is nondecreasing. The above analysis implies that $A^{1,R}_{t_1}>A^{1,R}_{t_2}$, which is a contradiction. Thus, the claim $A^2_T-A^2_t\leq A^1_T-A^1_t$ holds for any $t\in[0,T]$  and consequently, $Y^2_t\leq Y^1_t$. The proof is complete.
\end{proof}

Recall that for the single mean reflected case, the first component of the solution satisfying the Skorokhod condition  is minimal among all deterministic solutions when the coefficient is  deterministically linearly dependent on $y$ (see Theorem 11 in \cite{BEH} for the classical case and Proposition 3.9 in \cite{LiuW} for the $G$-expectation case). For the double mean reflected case, we have two Skorokhod conditions, which ensure both the minimality of the forces aiming to push the solution upwards and to push the solution downwards, respectively. In the following, we investigate the case if we only impose the minimality property for one of the forces described above, i.e., we consider the following two mean reflected BSDEs with single Skorokhod condition
\begin{equation}\label{nonlinearyz3}
\begin{cases}
Y_t=\xi+\int_t^T f(s,Y_s,Z_s)ds+\int_t^T g(s,Z_s)d\langle B\rangle_s-\int_t^T Z_s dB_s-(K_T-K_t)+A_T-A_t, \\
\hat{\mathbb{E}}[L(t,Y_t)]\leq 0\leq \hat{\mathbb{E}}[R(t,Y_t)], \\
A_t=A^R_t-A^L_t,\ A^R,A^L\in I[0,T], \ \int_0^T \hE[R(t,Y_t)]dA_t^R=0,
\end{cases}
\end{equation}
and
\begin{equation}\label{nonlinearyz2}
\begin{cases}
Y_t=\xi+\int_t^T f(s,Y_s,Z_s)ds+\int_t^T g(s,Z_s)d\langle B\rangle_s-\int_t^T Z_s dB_s-(K_T-K_t)+A_T-A_t, \\
\hE[L(t,Y_t)]\leq 0\leq \hE[R(t,Y_t)], \\
A_t=A^R_t-A^L_t, \ A^R,A^L\in I[0,T], \ \int_0^T \hE[L(t,Y_t)]dA^L_t=0.
\end{cases}
\end{equation}
Since \eqref{nonlinearyz3} only imposes the minimality condition on $A^R$ and \eqref{nonlinearyz2}  only imposes the minimality condition on $A^L$, it is natural to conjecture that the solution of BSDE with double mean reflections falls between the solutions of \eqref{nonlinearyz3} and \eqref{nonlinearyz2}.

\begin{proposition}\label{prop11}
Suppose that the coefficients $f,g$ satisfy Assumption \ref{assfg1}. 
 Let $\xi\in L^\beta(\Omega_T)$ and $L,R$ satisfy Assumption \ref{ass2} with $\hE[L(T,\xi)]\leq 0\leq \hE[R(T,\xi)]$. Let $(\underline{Y},\underline{Z},\underline{K},\underline{A})$, $(\bar{Y},\bar{Z},\bar{K},\bar{A})$, $(Y,Z,K,A)$ be the solutions to \eqref{nonlinearyz3}, \eqref{nonlinearyz2}, \eqref{nonlinearyz11}, respectively. Then, for any $t\in[0,T]$, we have $\underline{Y}_t\leq Y_t\leq \bar{Y}_t$.
\end{proposition}

\begin{proof}
We only prove the second inequality since the first one can be proved similarly. The proof is similar with the one for Proposition \ref{prop11'}.  Actually, what we need to do is to replace $Y^1,A^1,A^{1,R},A^{1,L}$ and $Y^2,A^2,A^{2,R},A^{2,L}$ in the proof of Proposition \ref{prop11'} by $\bar{Y},\bar{A},\bar{A}^{R},\bar{A}^{L}$ and $Y,A,A^{R},A^{L}$, respectively. The proof is complete.
\end{proof}

\begin{remark}
Proposition \ref{prop11'} and Proposition \ref{prop11} extends Proposition 4.6 and Proposition 4.4 in \cite{Li} to the $G$-expectation framework. What's more, since the constraints are made on the distribution of the solution, the comparison theorem for the general case may not hold true. Especially, the larger initial value may not induce the larger solution. A counterexample can be found in Example 3.10 in \cite{LiuW}.
\end{remark}

\subsection{Connection with optimization problems under model uncertainty}
Recalling that the expectation of the first component of doubly mean reflected BSDE in the classical case corresponds to the value function to some ``game" problem (see Theorem 3.6 in \cite{FS2} and Theorem 4.7 in \cite{Li}). More specifically, the value function is of the $\sup_q\inf_s$ ($=\inf_s\sup_q$) type, where $s,q$ are constants taking values from some bounded interval. However, under $G$-framework, due to the fact that $\hE[\cdot]$ is an upper expectation, the upper value function and the lower value function take the form of ``$\sup_q\inf_s\sup_{P}$" and ``$\inf_s\sup_q\inf_{P}$", respectively. Therefore, the ``game" problem under $G$-expectation is more complicated. We show that the expectation of the first component of the solution to a doubly mean reflected $G$-BSDE lies between the lower value and the upper value of some optimization problems. Especially, if the first component of the solution has no mean uncertainty, its expectation coincides with the lower value and upper value.

To be more precisely, for any fixed $1<\alpha<\beta$, let $(Y,Z,K,A)\in \mathfrak{S}^\alpha_G(0,T)\times BV[0,T]$ be the solution to the doubly mean reflected $G$-BSDE \eqref{nonlinearyz1}. We define
\begin{align*}
\widetilde{Y}_t:=\hE_t[\xi+\int_t^T f(s,Y_s,Z_s)ds+\int_t^T g(s,Y_s,Z_s)d\langle B\rangle_s].
\end{align*}
It is easy to check that $Y_t=\widetilde{Y}_t+(A_T-A_t)$. 
Since $A$ is deterministic, it follows that
\begin{align*}
A_T-A_t=\hE[Y_t]-\hE[\widetilde{Y}_t]=\hE[-\widetilde{Y}_t]-\hE[-Y_t].
\end{align*}
Consequently,  we have
\begin{align}\label{ywidetildey}
Y_t=\widetilde{Y}_t+\hE[Y_t]-\hE[\widetilde{Y}_t]=\widetilde{Y}_t+\hE[-\widetilde{Y}_t]-\hE[-Y_t].
\end{align}
Set
\begin{align*}
&\bar{r}^{\widetilde{Y}}(t,x):=\hE[R(t,x+\widetilde{Y}_t-\hE[\widetilde{Y}_t])],\  \bar{l}^{\widetilde{Y}}(t,x):=\hE[L(t,x+\widetilde{Y}_t-\hE[\widetilde{Y}_t])],\\
&\underline{r}^{\widetilde{Y}}(t,x):=\hE[R(t,x+\widetilde{Y}_t+\hE[-\widetilde{Y}_t])],\  \underline{l}^{\widetilde{Y}}(t,x):=\hE[L(t,x+\widetilde{Y}_t+\hE[-\widetilde{Y}_t])].
\end{align*}
By the proof of Lemma \ref{proofofass1}, for any $t\in[0,T]$, $\bar{r}^{\widetilde{Y}}(t,\cdot),\bar{l}^{\widetilde{Y}}(t,\cdot)$, $\underline{r}^{\widetilde{Y}}(t,\cdot),\underline{l}^{\widetilde{Y}}(t,\cdot)$ are continuous, strictly increasing and satisfy
\begin{align*}
&\lim_{x\rightarrow \infty}\bar{r}^{\widetilde{Y}}(t,x)=\lim_{x\rightarrow \infty}\bar{l}^{\widetilde{Y}}(t,x)=+\infty, \
\lim_{x\rightarrow -\infty}\bar{r}^{\widetilde{Y}}(t,x)=\lim_{x\rightarrow -\infty}\bar{l}^{\widetilde{Y}}(t,x)=-\infty,\\
&\lim_{x\rightarrow \infty}\underline{r}^{\widetilde{Y}}(t,x)=\lim_{x\rightarrow \infty}\underline{l}^{\widetilde{Y}}(t,x)=+\infty, \
\lim_{x\rightarrow -\infty}\underline{r}^{\widetilde{Y}}(t,x)=\lim_{x\rightarrow -\infty}\underline{l}^{\widetilde{Y}}(t,x)=-\infty.
\end{align*}
Hence, for any fixed $t\in[0,T]$, the equations $\bar{r}^{\widetilde{Y}}(t,x)=0$, $\bar{l}^{\widetilde{Y}}(t,x)=0$, $\underline{r}^{\widetilde{Y}}(t,x)=0$, $\underline{l}^{\widetilde{Y}}(t,x)=0$ have unique solutions, which are denoted by $\bar{r}_t$, $\bar{l}_t$, $\underline{r}_t$ and $\underline{l}_t$, respectively.


\begin{theorem}\label{theorem3.6}
Suppose that $(Y,Z,K,A)$ is the solution to the BSDE with double mean reflections \eqref{nonlinearyz1}. Then, for  any $t\in[0,T]$, we have
\begin{align}\label{inequality}
\inf_{s\in[t,T]}\sup_{q\in[t,T]}\underline{R}_t(s,q)\leq-\hE[-Y_t]\leq  \hE[Y_t]\leq \sup_{q\in[t,T]}\inf_{s\in[t,T]}\bar{R}_t(s,q),
\end{align}
where
\begin{align*}
&\underline{R}_t(s,q)=-\hE[-\underline{R}'_t(s,q)], \
\bar{R}_t(s,q)=\hE[\bar{R}'_t(s,q)],\\
&\bar{R}'_t(s,q)=\mathbb{I}_t^{s\wedge q}(f,g)-(K_{s\wedge q}-K_t)+\xi I_{\{s\wedge q=T\}}+\bar{r}_q I_{\{q< T,q\leq s\}}+\bar{l}_s I_{\{s<q\}},\\
&\underline{R}'_t(s,q)=\mathbb{I}_t^{s\wedge q}(f,g)-(K_{s\wedge q}-K_t)+\xi I_{\{s\wedge q=T\}}+\underline{r}_q I_{\{q< T,q\leq s\}}+\underline{l}_s I_{\{s<q\}}
\end{align*}
and, for any $0\leq u\leq v\leq T$,
\begin{align*}
\mathbb{I}_u^v(f,g):=\int_u^{v} f(r,Y_r,Z_r)dr+\int_u^{v} g(r,Y_r,Z_r)d\langle B\rangle_r.
\end{align*}
If for any $t\in[0,T]$, $\hE[Y_t]=-\hE[-Y_t]$, then we have
\begin{align}\label{equality}
\hE[Y_t]=\inf_{s\in[t,T]}\sup_{q\in[t,T]}\underline{R}_t(s,q)=   \sup_{q\in[t,T]}\inf_{s\in[t,T]}\bar{R}_t(s,q).
\end{align}
\end{theorem}

\begin{proof}
Fix $t\in[0,T]$. In order to prove \eqref{inequality}, it suffices to show that for any $\varepsilon>0$, there exist $s^\varepsilon_t,q^\varepsilon_t\in[t,T]$, such that for all $s,q\in[t,T]$,
\begin{equation}\label{equa3.2}
-\varepsilon+\underline{R}_t(s^\varepsilon_t,q)\leq-\hE[-Y_t]\leq  \hE[Y_t]\leq \bar{R}_t(s,q^\varepsilon_t)+\varepsilon.
\end{equation}
We only prove the first inequality since the last one can be proved similarly. By \eqref{ywidetildey} and the definition of $\underline{r}^{\widetilde{Y}},\underline{l}^{{\widetilde{Y}}}$, we have
\begin{align*}
&\hE[R(v,\widetilde{Y}_v+\hE[-\widetilde{Y}_v]+\underline{r}_v)]=0\leq \hE[R(v,Y_v)]=\hE[R(v,\widetilde{Y}_v+\hE[-\widetilde{Y}_v]-\hE[-Y_v])], \\
&\hE[L(v,\widetilde{Y}_v+\hE[-\widetilde{Y}_v]+\underline{l}_v)]=0\geq \hE[L(v,Y_v)]=\hE[L(v,\widetilde{Y}_v+\hE[-\widetilde{Y}_v]-\hE[-Y_v])],
\end{align*}
which implies that
\begin{align}\label{barrbarl}
\underline{r}_v\leq -\hE[-Y_v]\leq \underline{l}_v, \ v\in[0,T].
\end{align}
 Set
\begin{align}\label{410}
s^\varepsilon_t=\inf\{s>t:-\hE[-Y_s]\geq \underline{l}_s-\varepsilon\}\wedge T.
\end{align}
We claim that the first inequality in \eqref{equa3.2} holds for this $s^\varepsilon_t$. Indeed, since $-\hE[-Y_s]<\underline{l}_s-\varepsilon$ on $s\in(t,s^\varepsilon_t)$, the strict monotonicity of $\underline{l}^{\widetilde{Y}}$ implies that
\begin{align*}
\hE[L(s,Y_s)]=\hE[L(s,\widetilde{Y}_s+\hE[-\widetilde{Y}_s]-\hE[-Y_s])]<\hE[L(s,\widetilde{Y}_s+\hE[-\widetilde{Y}_s]+\underline{l}_s)]=0.
\end{align*}
Due to the fact that $\int_0^T \hE[L(s,Y_s)]dA^L_s=0$, we have $A^L_q-A^L_t=0$ for $q\in(t,s^\varepsilon_t)$ and thus $A_q-A_t=A^R_q-A^R_t\geq 0$ for $q\in(t,s^\varepsilon_t)$. By the continuity of $A$,  we deduce that
\begin{align}\label{equa3.2'}
A_q-A_t\geq 0, \ q\in(t,s^\varepsilon_t].
\end{align}
 It is easy to check that for $q\in(t,s^\varepsilon_t]$
\begin{align*}
\underline{R}_t(s^\varepsilon_t,q)=&-\hE[-\mathbb{I}_t^{q}(f,g)-\xi I_{\{q=T\}}+(K_q-K_t)]+\underline{r}_q I_{\{q<T\}}\\
\leq &-\hE[-\mathbb{I}_t^{q}(f,g)-\xi I_{\{q=T\}}+(K_q-K_t)]-\hE[-Y_q] I_{\{q<T\}}+A_q-A_t\\
\leq &-\hE[-Y_t-\int_t^q Z_s dB_s]=-\hE[-Y_t],
\end{align*}
where we have used \eqref{barrbarl}, \eqref{equa3.2'} in the first inequality, the sublinearity of $\hE[\cdot]$ and the dynamics of $Y$ in the second inequality. On the other hand, for $q\in(s^\varepsilon_t,T]$,  we have
\begin{align*}
\underline{R}_t(s^\varepsilon_t,q)=&-\hE[-\mathbb{I}_t^{s^\varepsilon_t}(f,g) +(K_{s^\varepsilon_t}-K_t)]+\bar{l}_{s^\varepsilon_t}\\
\leq &-\hE[-\mathbb{I}_t^{s^\varepsilon_t} (f,g)+(K_{s^\varepsilon_t}-K_t)]-\hE[-Y_{s^\varepsilon_t} ] +\varepsilon+A_{s^\varepsilon_t} -A_t\\
\leq &-\hE[-Y_t-\int_t^{s^\varepsilon_t}  Z_s dB_s]+\varepsilon=-\hE[-Y_t]+\varepsilon,
\end{align*}
where we have used \eqref{equa3.2'} and definition of $s^\varepsilon_t$ in the first inequality, the sublinearity of $\hE[\cdot]$ and the dynamics of $Y$ in the second inequality. Therefore, all the above analysis yields the first inequality in \eqref{equa3.2}. Set
\begin{align}\label{411}
q^\varepsilon_t=\inf\{s>t: \hE[Y_s]\leq \bar{r}_s+\varepsilon\}\wedge T.
\end{align}
By a similar analysis as above, we could obtain the last inequality in \eqref{equa3.2}.

It remains to show \eqref{equality}. To this end, we only need to prove that for any $s,q\in[t,T]$, $\underline{R}_t(s,q)=\bar{R}_t(s,q)$. Indeed, it follows from \eqref{ywidetildey} that $\hE[\widetilde{Y}_t]=-\hE[-\widetilde{Y}_t]$, $t\in[0,T]$.  By the definition of $\underline{r}_t$, $\underline{l}_t$, $\bar{r}_t$ and $\bar{l_t}$, we have $\bar{r}_t=\underline{r}_t$ and $\bar{l_t}=\underline{l}_t$, $t\in[0,T]$. Consequently, we obtain that $\bar{R}'_t(s,q)=\underline{R}'_t(s,q)$. Note that
\begin{align*}
\mathbb{I}_t^{s\wedge q}(f,g)-(K_{s\wedge q}-K_t)+\xi I_{\{s\wedge q=T\}}=Y_t-Y_{s\wedge q} I_{\{s\wedge q<T\}}+\int_t^{s\wedge q}Z_s dB_s.
\end{align*}
Since $Y_t$, $Y_{s\wedge q}$ and $\int_t^{s\wedge q}Z_s dB_s$ have no mean uncertainty, $\bar{R}'_t(s,q)$ has no mean uncertainty.  The proof is complete.
\end{proof}

Moreover, if the coefficient $f$ is  deterministically linearly dependent on $y$ and the loss functions $L,R$ take the following form
\begin{align*}
    L(t,x)=x-L_t, \ R(t,x)=x-R_t,
\end{align*} 
we could connect the $G$-expectation of the solution with another form of optimization problem like what follows.
\begin{theorem}\label{myw405}
Let $(Y, Z, K,A)$ be the solution to the following mean reflected $G$-BSDE with double mean reflections
\begin{equation}\label{503}
\begin{cases}
Y_t=\xi+\int_t^T (\gamma_s Y_s +f_s) d s-\int_t^T Z_s d B_s-(K_T-K_t)+(A_T-A_t), \\
R_t\leq \hE[Y_t]\leq L_t, \\
A_t=A^R_t-A^L_t, \ A^R,A^L\in I[0,T] \textrm{ and } \int_0^T( \hE[Y_t]-R_t) dA_t^R=\int_0^T( \hE[Y_t]-L_t) dA^L_t=0.
\end{cases}
\end{equation}
where $\{f_t\}_{t\in[0,T]}, \{\gamma_t\}_{t\in[0,T]}$ are bounded deterministic measurable functions and $\{L_t\}_{t\in[0,T]},\{R_t\}_{t\in[0,T]}$ are deterministic continuous functions with $\inf_{t\in[0,T]}(L_t-R_t)>0$. 
Then for all $t \in[0, T]$, we have
\begin{align}\label{503'}
\hat{\mathbb{E}}\left[Y_t\right]=\sup _{q \in[t, T]}\inf _{s \in[t, T]}\hat{\mathbb{E}}\left[y_t^{s\wedge q}\right]=\inf _{s \in[t, T]}\sup _{q\in[t, T]}\hat{\mathbb{E}}\left[y_t^{s\wedge q}\right],
\end{align}
 where $y^{s\wedge q}$ is the first component of the solution to the following $G$-BSDE on the time horizon $[0,s\wedge q]$:
 $$
y_t^{s\wedge q}=\left[\xi I_{\{s\wedge q=T\}}+L_s I_{\{s<q\}}+R_q I_{\{q\leq s<T\}}\right]+\int_t^{s\wedge q}( \gamma_r y_r^{s\wedge q} +f_r)d r-\int_t^{s\wedge q} z_r^{\tau\wedge\sigma} d B_r-(k^{s\wedge q}_{s\wedge q}-k^{s\wedge q}_t),
$$
and the saddle-point $s^*,q^*\in[t, T]$ is given by:
\begin{align}\label{502}
s^*=\inf \left\{s \geq t: \hat{\mathbb{E}}\left[Y_s\right]=L_s\right\} \wedge T,\ q^*=\inf \left\{s \geq t: \hat{\mathbb{E}}\left[Y_s\right]=R_s\right\} \wedge T.
\end{align}
\end{theorem}

\begin{remark}
    Under the setting described in Theorem \ref{myw405} and recalling the notations in Theorem \ref{theorem3.6}, we have
    \begin{align*}
        \bar{R}_t(s,q)=\hE[\xi I_{\{s\wedge q=T\}}+L_s I_{\{s<q\}}+R_q I_{\{q\leq s<T\}}+\int_t^{s\wedge q}(\gamma_rY_r+f_r)dr-(K_{s\wedge q}-K_t)].
    \end{align*}
    The term $\bar{R}_t(s,q)$ is usually not equal to $\hE[y_t^{s\wedge q}]$. Therefore, the optimization problem in \eqref{503'} differs from the one in \eqref{inequality}.
\end{remark}

\begin{proof}
First, we show that $\hat{\mathbb{E}}[Y_t]=\hat{\mathbb{E}}[y_t^{s^*\wedge q^*}]$. Using the notations in the proof of Theorem \ref{thm3.8}, $(Y^a,Z^a,K^a,A^a)$ solves the doubly mean reflected $G$-BSDE with terminal value $\xi^a$, loss functions $L^a,R^a$ and coefficient $\{e^{a_s}f_s\}_{s\in[0,T]}$. Besides, the constraining processes satisfy
\begin{align*}
    A^{a,L}_t=\int_0^t e^{a_r}dA^L_r, \ A^{a,R}_t=\int_0^t e^{a_r}dA^R_r.
\end{align*}
Recalling the definition of $s^*,q^*$ in \eqref{502}, it is easy to check that $dA^L_r=dA^R_r=0$ for $r\in[t,s^*\wedge q^*]$, which implies that $dA^{a,L}_r=dA^{a,R}_r=0$ and thus
\begin{align*}
    Y^a_t=Y^a_{s^*\wedge q^*}+\int_t^{s^*\wedge q^*} e^{a_r}f_rdr-\int_t^{s^*\wedge q^*} Z^a_rdB_r-(K^a_{s^*\wedge q^*}-K^a_t).
\end{align*}
Consequently, we have
\begin{align*}
& \hat{\mathbb{E}}\left[Y^a_t\right]=\hat{\mathbb{E}}\left[Y^a_{s^*\wedge q^*}+\int_t^{s^*\wedge q^*} e^{a_r}f_rdr\right]\\
=&\hat{\mathbb{E}}\left[{Y}^a_{q^*} I_{\left\{q^* \leq s^*<T\right\}}+{Y}^a_{s^*} I_{\left\{s^*<q^*\right\}}+{\xi}^a I_{\left\{s^* \wedge q^*=T\right\}}+\int_t^{s^* \wedge q^*} e^{a_r} f_r d r\right] \\
=&\hat{\mathbb{E}}\left[e^{a_T}\xi I_{\left\{s^* \wedge q^*=T\right\}}\right]+ e^{a_{q^*}}R_{q^*} I_{\left\{q^* \leq s*<T\right\}}+e^{a_{s^*}} L_{s^*} I_{\left\{s^*<q^*\right\}}+\int_t^{s^* \wedge q^*} e^{a_r} f_r d r=e^{a_t}\hat{\mathbb{E}}\left[y_t^{s^* \wedge q^*}\right],
\end{align*}
which implies that $\hat{\mathbb{E}}[Y_t]=\hat{\mathbb{E}}[y_t^{s^*\wedge q^*}]$.

Now let $s \in[t, T]$. Then it follows from the Skorokhod condition that ${A}^{a,R}_{s \wedge q^*}-A^{a,R}_t=0$. Therefore, we have
\begin{align*}
    Y^a_t=Y^a_{s\wedge q^*}+\int_t^{s\wedge q^*} e^{a_r}f_rdr-\int_t^{s\wedge q^*} Z^a_rdB_r-(K^a_{s\wedge q^*}-K^a_t)-({A}^{a,L}_{s \wedge q^*}-A^{a,L}_t).
\end{align*}
Noting that $A^{a,L}$ is nondecreasing, we obtain that 
\begin{align*}
& \hat{\mathbb{E}}\left[Y^a_t\right]=\hat{\mathbb{E}}\left[Y^a_{s\wedge q^*}+\int_t^{s\wedge q^*} e^{a_r}f_rdr-({A}^{a,L}_{s \wedge q^*}-A^{a,L}_t)\right]\\
\leq &\hat{\mathbb{E}}\left[{Y}^a_{q^*} I_{\left\{q^* \leq s<T\right\}}+{Y}^a_{s} I_{\left\{s<q^*\right\}}+{\xi}^a I_{\left\{s \wedge q^*=T\right\}}+\int_t^{s \wedge q^*} e^{a_r} f_r d r\right] \\
=&\hat{\mathbb{E}}\left[e^{a_T}\xi I_{\left\{s \wedge q^*=T\right\}}\right]+ e^{a_{q^*}}R_{q^*} I_{\left\{q^* \leq <T\right\}}+e^{a_{s}} L_{s} I_{\left\{s<q^*\right\}}+\int_t^{s \wedge q^*} e^{a_r} f_r d r=e^{a_t}\hat{\mathbb{E}}\left[y_t^{s \wedge q^*}\right]. 
\end{align*}
The above analysis indicates that 
\begin{align*}
    \hE[Y_t]\leq \inf_{s\in[t,T]}\hE\left[y_t^{s \wedge q^*}\right]\leq \sup_{q\in[t,T]}\inf_{s\in[t,T]}\hE\left[y_t^{s \wedge q}\right].
\end{align*}
Similarly, we have
\begin{align*}
    \hE[Y_t]\geq \sup_{q\in[t,T]}\hE\left[y_t^{s^* \wedge q}\right]\geq \inf_{s\in[t,T]}\sup_{q\in[t,T]}\hE\left[y_t^{s \wedge q}\right].
\end{align*}
It is obvious that $\sup_{q\in[t,T]}\inf_{s\in[t,T]}\hE\left[y_t^{s \wedge q}\right]\leq \inf_{s\in[t,T]}\sup_{q\in[t,T]}\hE\left[y_t^{s \wedge q}\right]$. Hence the desired result holds.
\end{proof}



\appendix
\section{Basic properties of $G$-BSDEs with Lipschitz coefficients}

In this appendix, we state the well-posedness and some basic estimates of $G$-BSDEs for reader's convenience, more relevant details can be found in \cite{HJPS1}.
\begin{theorem}[\cite{HJPS1}]\label{mywA1}
The $G$-BSDE with terminal condition $\xi\in L_{G}^{\beta}\left(\Omega_{T}\right)$ and generators $f, g$ satisfying Assumption \ref{assfg} for some $\beta>1$  has a unique solution $(Y, Z, K)\in \mathfrak{S}_{G}^{\alpha}(0, T)$ for any $1<\alpha <\beta$ and $ t \longmapsto Y_{t}$ is continuous. Moreover, there exists a constant $\widetilde{C}(\alpha, T, G, \kappa)>0$ such that
\begin{gather}
\label{AA1} \left|{Y}_{t}\right|^{\alpha}\leq \widetilde{C}(\alpha, T, G, \kappa)\hat{\mathbb{E}}_{t}\bigg[|{\xi}|^{\alpha}+\int_{t}^{T} |{h}_{s}|^{\alpha}d s\bigg],\\
\label{AA2} \hat{\mathbb{E}}\bigg[\bigg(\int_{0}^{T}\left|{Z}_{s}\right|^{2} d s\bigg)^{\frac{\alpha}{2}}\bigg]+\hat{\mathbb{E}}\left[|K_T|^{\alpha}\right]\leq \widetilde{C}(\alpha, T, G, \kappa)\bigg\{\|{Y}\|_{S_{G}^{\alpha}}^{\alpha}+\Big\|\int_{0}^{T} h_{s} d s\Big\|_{L_{G}^{\alpha}}^{\alpha}\bigg\},
\end{gather}
where $h_s=|f(s, 0,0)|+| g(s, 0,0) |$.
\end{theorem}
\begin{remark}\label{mywA2}
Actually, if $\xi\in L_{G}^{\beta}\left(\Omega_{T}\right)$ and $f, g$ satisfy Assumption \ref{assfg} for some $\beta>1$, we could get $Y\in M_{G}^{\beta}\left(0, T \right)$, which was verified in Lemma 3.2 of \cite{L}.
\end{remark}
\begin{theorem}[\cite{HJPS1}]\label{mywA3}
Let $\xi^{l} \in L_{G}^{\beta}\left(\Omega_{T}\right)$, $l=1,2$, and $f^{l}, g^{l}$ satisfy Assumption \ref{assfg} for some $\beta>1$. Assume that $\left(Y^{l}, Z^{l}, K^{l}\right) \in \mathfrak{S}_{G}^{\alpha}(0, T)$ for some $1<\alpha <\beta$ is the solution of $G$-BSDE corresponding to the data $(\xi^{l}, f^{l}, g^{l}).$ Set $\hat{Y}_{t}=Y_{t}^{1}-Y_{t}^{2}$, $\hat{Z}_{t}=Z_{t}^{1}-Z_{t}^{2} $. Then there exists a positive constant $\widetilde{C}(\alpha, T, G, \kappa)$ such that
\begin{gather}
\label{AA3}\big|\hat{Y}_{t}\big|^{\alpha}\leq \widetilde{C}(\alpha, T, G, \kappa)\hat{\mathbb{E}}_{t}\bigg[\Big(|\hat{\xi}|+\int_{t}^{T} |\hat{h}_{s}|d s\Big)^{\alpha}\bigg],\\
\label{AA4} \hat{\mathbb{E}}\bigg[\bigg(\int_{0}^{T}\big|\hat{Z}_{s}\big|^{2} d s\bigg)^{\frac{\alpha}{2}}\bigg] \leq \widetilde{C}(\alpha, T, G, \kappa)\Bigg\{\|\hat{Y}\|_{S_{G}^{\alpha}}^{\alpha}+\|\hat{Y}\|_{S_{G}^{\alpha}}^{\frac{\alpha}{2}} \sum_{l=1}^{2}\bigg[\left\|Y^{l}\right\|_{S_{G}^{\alpha}}^{\frac{\alpha}{2}}+\Big\|\int_{0}^{T} h_{s}^{l, 0} d s\Big\|_{L_{G}^{\alpha}}^{\frac{\alpha}{2}}\bigg]\Bigg\},
\end{gather}
where $h_{s}^{l, 0}=|f^{l}(s, 0,0)|+| g^{l}(s, 0,0) |$, $\hat{\xi}=\xi^{1}-\xi^{2}$, and $\hat{h}_{s}=\left|f^{1}\left(s, Y_{s}^{2}, Z_{s}^{2}\right)-f^{2}\left(s, Y_{s}^{2}, Z_{s}^{2}\right)\right|+\left|g^{1}\left(s, Y_{s}^{2}, Z_{s}^{2}\right)-g^{2}\left(s, Y_{s}^{2}, Z_{s}^{2}\right)\right|$. Moreover, according to the proof of Proposition 5.1 in \cite{HJPS1}, we could write $\widetilde{C}(\alpha, T, G, \kappa)$ in \eqref{AA3} in a more accurate form $\widetilde{C}(\alpha, T, G, \kappa)=\widetilde{C}(\alpha,G,\kappa)e^{\widetilde{C}(\alpha,G,\kappa)T}$.
\end{theorem}



\begin{thebibliography}{00}
\bibitem{BL} Bai, X., Lin, Y. (2014) On the existence and uniqueness of solutions to stochastic differential equations driven by $G$-Brownian motion with integral-Lipschitz coefficients.  Acta Mathematicae Applicatae Sinica, English Series, 30(3), 589-610.
\bibitem{B1} Bouchard, B., Elie, R.,  R\'{e}veillac, A. (2015) BSDEs with weak terminal condition. The Annals of Probability, 43(2), 572-604.

\bibitem{BCGL} Briand, P., Chaudru de Raynal, P.-\'{E}., Guillin, A. and Labart, C. (2020) Particles systems and numerical schemes for mean reflected stochastic differential equations.  Ann. Appl. Probab., 30(4): 1884-1909.


\bibitem{BEH} Briand, P., Elie, R. and Hu, Y. (2018) BSDEs with mean reflection.  Ann. Appl. Probab., 28(1): 482-510.
\bibitem{BH} Briand, P.,  Hibon, H. (2021) Particles systems for mean reflected BSDEs. Stochastic Processes and their Applications, 131, 253-275.


\bibitem{CK} Cvitanic, J., Karatzas, I. (1996) Backward stochastic differential equations with reflection and Dynkin games. Ann. Probab.,  24(4): 2024-2056.
\bibitem {DHP11} Denis, L., Hu, M., Peng, S. (2011) Function spaces and capacity related to a sublinear expectation: application to $G$-Brownian motion paths. Potential Anal., 34: 139-161.
\bibitem{DE} Djehiche, B., Elie, R.,  Hamad\`ene, S. (2021) Mean-field reflected backward stochastic differential equations. Ann. Appl. Probab., in press.
\bibitem{DQS} Dumitrescu, R., Quenez, M.C., Sulem, A. (2016) Generalized Dynkin games and doubly reflected BSDEs with jumps. Electronic Journal of Probability, 21: 1-32.
\bibitem{E} El Karoui, N., Kapoudjian, C., Pardoux, E., Peng, S.,  Quenez, M. C. (1997) Reflected solutions of backward SDE's, and related obstacle problems for PDE's. the Annals of Probability, 25(2), 702-737.
\bibitem{EJ} Epstein, L. G.,  Ji, S. (2014) Ambiguous volatility, possibility and utility in continuous time. Journal of Mathematical Economics, 50, 269-282.

\bibitem{FS2} Falkowski, A. and Slomi\'{n}ski, L. (2022) Backward stochastic differential equations with mean reflection and two constraints. Bulletin des Sciences Math\'{e}matiques, 176: 103117.

\bibitem{G} Gao, F. (2009) Pathwise properties and homeomorphic flows for stochastic differential equations driven by $G$-Brownian motion. Stochastic Process. Appl., 119: 3356-3382.
\bibitem{GIOQ} Grigorova, M., Imkeller, P., Ouknine, Y., Quenez, M.C. (2018) Doubly reflected BSDEs and $\mathcal{E}^f$-Dynkin games: beyond the right-continuous case. Electronic Journal of Probability,  23: 1-38.
\bibitem{GLX} Gu, Z., Lin, Y.,  Xu, K. (2023) Quadratic BSDEs with mean reflection driven by $G$-Brownian motion. Stochastics and Dynamics, 23(05), 2350044.
\bibitem{H} He, W. (2022) BSDEs driven by $G$-Brownian motion with non-Lipschitz coefficients. Journal of Mathematical Analysis and Applications, 505(2), 125569.
\bibitem{HE} He, W. (2024) Multi-dimensional mean-reflected BSDEs driven by $G$-Brownian motion with time-varying non-Lipschitz coefficients. Statistics \& Probability Letters, 206, 109977.
\bibitem{HHL} Hibon, H., Hu, Y., Lin, Y., Luo, P. and Wang, F. (2018) Quadratic BSDEs with mean reflection. Math. Control Relat. Fields, 8(3-4), 721-738.
\bibitem{HJPS1} Hu, M., Ji, S., Peng, S., Song, Y. (2014) {Backward stochastic differential equations driven by $G$-Brownian motion}. Stochastic Processes and their Applications, 124: 759-784.
\bibitem{HJb} Hu, M., Ji, S., Peng, S., Song, Y. (2014) Comparison theorem, Feynman-Kac formula and Girsanov transformation for BSDEs driven by $G$-Brownian motion.  Stochastic Processes and their Applications, 124(2), 1170-1195.
\bibitem{HWZ} Hu, M., Wang, F., Zheng, G. (2016) Quasi-continuous random variables and processes under the $G$-expectation framework. Stochastic Processes and their Applications, 126: 2367-2387.
\bibitem{HH} Hu, Y., Huang, J., Li, W. (2022) Backward stochastic differential equations with conditional reflection and related recursive optimal control problems. arXiv preprint arXiv:2211.07191.
\bibitem{HMW} Hu, Y., Moreau, R., Wang, F. (2024) General mean reflected backward stochastic differential equations. J Theor Probab 37, 877-904.
\bibitem{HT1} Hu, Y., Tang, S.,  Xu, Z. (2022) Optimal control of SDEs with expected path constraints and related constrained FBSDEs. arXiv preprint arXiv:2201.00321.
\bibitem{Li2} Li, H. (2023) The Skorokhod problem with two nonlinear constraints. Probability and Mathematical Statistics, 43(2), 207-239.
\bibitem{Li} Li, H. (2024) Backward stochastic differential equations with double mean reflections. Stochastic Processes and their Applications, 104371.
\bibitem{LN} Li, H.,  Ning, N. (2024) Propagation of chaos for doubly mean reflected BSDEs. arxiv preprint arxiv:2401.16617.
\bibitem{LP'} Li, H., Peng, S. (2020) Reflected backward stochastic differential equations driven by $G$-Brownian motion with an upper obstacle. Stochastic Processes and their Applications, 130, 6556-6579.
\bibitem{LP} Li, H., Peng, S., \& Soumana Hima, A. (2018) Reflected solutions of backward stochastic differential equations driven by $G$-Brownian motion. Science China Mathematics, 61(1), 1-26.
   \bibitem{LPS} Li, H., Peng, S., Song, Y. (2018) Supermartingale decomposition theorem under $G$-expectation. Electron. J. Probab., 23: 1-20.
\bibitem{LS'} Li, H., Song, Y. (2021) Backward stochastic differential equations driven by $G$-Brownian motion with double reflections. J. Theor. Probab., 34, 2285-2314.


   \bibitem{lp} Li, X., Peng, S. (2011) Stopping times and related It\^{o}'s calculus with $G$-Brownian motion. Stochastic Process. Appl., 121: 1492-1508.




\bibitem{L} Liu, G. (2020) Multi-dimensional BSDEs driven by $G$-Brownian motion and related system of fully nonlinear PDEs. Stochastics, 92(5), 659-683.
\bibitem{LiuW} Liu, G., Wang, F. (2019) BSDEs with mean reflection driven by $G$-Brownian motion. Journal of Mathematical Analysis and Applications, 470: 599-618.
\bibitem{LS} Lu, H., \& Song, Y. (2021) Forward-backward stochastic differential equations driven by $G$-Brownian motion. arXiv preprint arXiv:2104.06868.
\bibitem{Mao} Mao, X. (1995) Adapted solutions of backward stochastic differential equations with non-lipschitz coefficients. Stochastic Processes and their Applications, 58(2), 281-292.
\bibitem {P07a} Peng, S. (2007) $G$-expectation, $G$-Brownian motion and related stochastic calculus of It\^o type. Stochastic analysis and
applications,  Abel Symp., 2: 541-567. Springer, Berlin.

\bibitem {P08a} Peng, S. (2008) Multi-dimensional $G$-Brownian motion and related stochastic calculus under $G$-expectation. Stochastic Processes and their Applications, 118(12): 2223-2253.

\bibitem{P19} Peng, S. (2019) Nonlinear Expectations and Stochastic Calculus Under Uncertainty: With Robust CLT and $G$-Brownian Motion. Probability Theory and Stochastic Modelling 95, Springer.
\bibitem{QW} Qu, B., Wang, F. (2023) Multi-dimensional BSDEs with mean reflection. Electronic Journal of Probability, 28, 1-26.

\bibitem{STZ1} Soner, H. M., Touzi, N.,  Zhang, J. (2011) Martingale representation theorem for the $G$-expectation. Stochastic Processes and their Applications, 121(2), 265-287.
\bibitem{STZ} Soner, H. M., Touzi, N., Zhang, J. (2012) Wellposedness of second order backward SDEs. Probability Theory and Related Fields, 153(1-2), 149-190.


\bibitem{Song} Song, Y. (2011) Some properties on $G$-evaluation and its applications to
$G$-martingale decomposition, Science China Mathematics, 54(2), 287-300.
\bibitem{WY} Watanabe, S.,  Yamada, T. (1971) On the uniqueness of solutions of stochastic differential equations II. Journal of Mathematics of Kyoto University, 11(3), 553-563.
\bibitem{YW} Yamada, T.,  Watanabe, S. (1971) On the uniqueness of solutions of stochastic differential equations. Journal of Mathematics of Kyoto University, 11(1), 155-167.
























































\end{thebibliography}


\section*{Acknowledgments}
Li's work was supported by the National Natural Science Foundation of China (No. 12301178), the Natural Science Foundation of Shandong Province for Excellent Young Scientists Fund Program (Overseas) (No. 2023HWYQ-049), the Natural Science Foundation of Shandong Province (No. ZR2023ZD35) and the Qilu Young Scholars Program of Shandong University.

\end{document}